\documentclass[11pt,a4paper]{article}
\usepackage{mathrsfs,amsthm,graphicx,xcolor,verbatim,bbm,amsmath,amsfonts,amssymb,newclude,nicefrac,graphicx,enumitem,hyperref,bm,mathtools,xparse,etoolbox,listings}
\usepackage[margin=1in]{geometry}
\usepackage[capitalise]{cleveref}

\crefname{enumi}{item}{items}

\crefname{equation}{}{}
\crefname{subsection}{Subsection}{Subsections}

\hypersetup{
    colorlinks,
    linkcolor={red!80!black},
    citecolor={green},
    urlcolor={blue!80!black}
}

\theoremstyle{plain}
\newtheorem{theorem}{Theorem}[section]
\newtheorem{lemma}[theorem]{Lemma}
\newtheorem{prop}[theorem]{Proposition}
\newtheorem{cor}[theorem]{Corollary}

\newtheorem{setting}[theorem]{Setting}

\theoremstyle{definition}

\DeclareMathAlphabet{\mathpzc}{OT1}{pzc}{m}{it}

\DeclareFontEncoding{LS1}{}{}
\DeclareFontSubstitution{LS1}{stix}{m}{n}
\DeclareMathAlphabet{\mathscr}{LS1}{stixscr}{m}{n}

\newcommand{\E}{\mathbb{E}}
\renewcommand{\P}{\mathbb{P}}

\newcommand{\R}{\mathbb{R}}
\newcommand{\N}{\mathbb{N}}

\newcommand{\bbF}{\mathbb{F}}
\newcommand{\bbX}{\mathbb{X}}
\newcommand{\bbY}{\mathbb{Y}}

\newcommand{\w}[1]{\mathfrak{w}^{#1}}
\renewcommand{\b}[1]{\mathfrak{b}^{#1}}
\renewcommand{\v}[1]{\mathfrak{v}^{#1}}
\renewcommand{\c}[1]{\mathfrak{c}^{#1}}

\newcommand{\smallsum}{\textstyle\sum}

\newcommand{\with}{\curvearrowleft}

\newcommand{\cB}{\mathcal{B}}

\newcommand{\cF}{\mathcal{F}}
\newcommand{\cG}{\mathcal{G}}

\newcommand{\cL}{\mathcal{L}}

\newcommand{\cX}{\mathcal{X}}
\newcommand{\cY}{\mathcal{Y}}

\newcommand{\bfa}{\mathbf{a}}

\newcommand{\bfe}{\mathbf{e}}

\newcommand{\bfA}{\mathbf{A}}

\newcommand{\bfL}{\mathbf{L}}

\newcommand{\scrC}{\mathscr{C}}

\newcommand{\scrL}{\mathscr{L}}

\newcommand{\scrN}{\mathscr{N}}

\newcommand{\fC}{\mathfrak{C}}
\newcommand{\fD}{\mathfrak{D}}

\newcommand{\fG}{\mathfrak{G}}

\newcommand{\fL}{\mathfrak{L}}
\newcommand{\fM}{\mathfrak{M}}

\newcommand{\fb}{\mathfrak{b}}
\newcommand{\fc}{\mathfrak{c}}
\newcommand{\fd}{\mathfrak{d}}

\newcommand{\fm}{\mathfrak{m}}
\newcommand{\fn}{\mathfrak{n}}

\newcommand{\fv}{\mathfrak{v}}
\newcommand{\fw}{\mathfrak{w}}

\DeclarePairedDelimiter{\norm}{\lVert}{\rVert}
\DeclarePairedDelimiter{\abs}{\lvert}{\rvert}
\DeclarePairedDelimiter{\rbr}{(}{)}
\DeclarePairedDelimiter{\br}{[}{]}
\DeclarePairedDelimiter{\cu}{\{}{\}}

\newcommand{\Rect}{\mathfrak{R}}

\renewcommand{\d}{ \mathrm{d}}

\newcommand{\indicator}[1]{\mathbbm{1}_{\smash{#1}}}

\newcommand{\realization}[1] {\mathscr{N} ^{ #1  }}
\newcommand{\realapprox}[2]{\mathscr{N} ^{#1}_ {#2}}

\newcommand{\width}{H}

\newcommand{\supgn}{\mathfrak{g}}

\ExplSyntaxOn

\bool_new:N \g_noteobserve

\NewDocumentCommand{\nobs}{}{
  \bool_if:nTF { \g_noteobserve } {
    \bool_gset_false:N \g_noteobserve 
    note~
  } {
    \bool_gset_true:N \g_noteobserve 
    observe~
  }
}

\NewDocumentCommand{\Nobs}{}{
  \bool_if:nTF { \g_noteobserve } {
    \bool_gset_false:N \g_noteobserve 
    Note~
  } {
    \bool_gset_true:N \g_noteobserve 
    Observe~
  }
}

\seq_new:N \g_cflist_loaded
\seq_new:N \g_cflist_pending

\NewDocumentCommand{\cfadd}{ m }
{
  \seq_if_in:NnF \g_cflist_loaded { #1 } {
    \seq_if_in:NnF \g_cflist_pending { #1 } {
      \seq_gput_right:Nn \g_cflist_pending { #1 }
    }
  }
}

\NewDocumentCommand{\cfconsiderloaded}{ m }{
  \seq_gput_right:Nn \g_cflist_loaded {#1}
}

\NewDocumentCommand{\cfremove}{ m }
{
  \seq_gremove_all:Nn \g_cflist_pending { #1 }
}

\NewDocumentCommand{\cfload}{ o }
{
  \seq_if_empty:NTF \g_cflist_pending {\unskip} {
    (cf.\ \cref{\seq_use:Nn \g_cflist_pending {,}})\IfValueTF{#1}{#1~}{\unskip}
    \seq_gconcat:NNN \g_cflist_loaded \g_cflist_loaded \g_cflist_pending
    \seq_gclear:N \g_cflist_pending
  }
}

\NewDocumentCommand{\cfclear} {} {
  \seq_gclear:N \g_cflist_loaded
  \seq_gclear:N \g_cflist_pending
}

\NewDocumentCommand{\cfout}{ o }
{
  \seq_if_empty:NTF \g_cflist_pending {\unskip} {
    (cf.\ \cref{\seq_use:Nn \g_cflist_pending {,}})\IfValueTF{#1}{#1~}{\unskip}
    \seq_gclear:N \g_cflist_pending
  }
}

\NewDocumentCommand{\ifnocf} { m } {
  \seq_if_empty:NT \g_cflist_pending { #1 }
}

\ExplSyntaxOff

\NewDocumentEnvironment{cproof}{m}
{\begin{proof}[Proof of \cref{#1}]}%
{\noindent The proof of \cref{#1} is thus complete.
\end{proof}}

\NewDocumentEnvironment{cproof2}{m}
{\begin{proof}[Proof of \cref{#1}]}%
{\noindent This completes the proof of \cref{#1}.
\end{proof}}

\usepackage[utf8]{inputenc}

\definecolor{codegreen}{rgb}{0,0.6,0}
\definecolor{codegray}{rgb}{0.5,0.5,0.5}
\definecolor{codepurple}{rgb}{0.58,0,0.82}
\definecolor{backcolour}{rgb}{0.95,0.95,0.92}

\lstdefinestyle{mystyle}{
    backgroundcolor=\color{backcolour},   
    commentstyle=\color{codegreen},
    keywordstyle=\color{magenta},
    numberstyle=\tiny\color{codegray},
    stringstyle=\color{codepurple},
    basicstyle=\ttfamily\footnotesize,
    breakatwhitespace=false,         
    breaklines=true,                 
    captionpos=b,                    
    keepspaces=true,                 
    numbers=left,                    
    numbersep=5pt,                  
    showspaces=false,                
    showstringspaces=false,
    showtabs=false,                  
    tabsize=2
}

\title{\vspace{-0.5cm}
A proof of convergence for stochastic gradient descent\\
in the training of artificial neural networks with\\
ReLU activation for constant target functions}
\author{
Arnulf Jentzen$^{1, 2}$ and
Adrian Riekert$^3$
\bigskip
\\
\small{$^1$ Faculty of Mathematics and Computer Science, University of M{\"u}nster,}
\vspace{-0.1cm}\\
\small{M{\"u}nster, Germany, e-mail: \texttt{ajentzen}\textcircled{\texttt{a}}\texttt{uni-muenster.de}}
\smallskip
\\
\small{$^2$ School of Data Science and Shenzhen Research Institute of Big Data,}
\vspace{-0.1cm}\\
\small{The Chinese University of Hong Kong, Shenzhen, China, e-mail: \texttt{ajentzen}\textcircled{\texttt{a}}\texttt{cuhk.edu.cn}}
\smallskip
\\
\small{$^3$ Faculty of Mathematics and Computer Science, University of M{\"u}nster,}
\vspace{-0.1cm}\\
\small{M{\"u}nster, Germany, e-mail: \texttt{ariekert}\textcircled{\texttt{a}}\texttt{uni-muenster.de}}
}
\date{\today}

\begin{document}

\maketitle
\begin{abstract}
In this article we study the stochastic gradient descent (SGD) optimization method 
in the training of fully-connected feedforward artificial neural networks 
with ReLU activation. The main result of this work proves that the risk of 
the SGD process converges to zero if the target function under consideration 
is constant. In the established convergence result the considered artificial 
neural networks consist of one input layer, one hidden layer, and one output 
layer (with $d \in \N$ neurons on the input layer, $\width \in \N$ neurons on the hidden layer, and one neuron on the output layer). The learning rates of the SGD process are assumed to be sufficiently small 
and the input data used in the SGD process to train the artificial neural networks 
is assumed to be independent and identically distributed.
\end{abstract}

\tableofcontents

\section{Introduction}
Artificial neural networks (ANNs) are these days widely used in several real world applications, including, e.g., text analysis, image recognition, autonomous driving, and game intelligence.
Stochastic gradient descent (SGD) optimization methods provide
the standard schemes which are used for the training 
of ANNs. Nonetheless, until today, 
there is no complete mathematical analysis 
in the scientific literature 
which rigorously explains the success of SGD optimization methods 
in the training of ANNs in numerical simulations.

 However, there are several interesting directions of research regarding the mathematical analysis of SGD optimization methods in the training of ANNs.
 The convergence of SGD optimization schemes for convex target functions is well understood, cf., e.g., \cite{BachMoulines2013, BachMoulines2011, Nesterov1983, Nesterov2014, Rakhlin2012} and the references mentioned therein.
For abstract convergence results for SGD optimization methods without convexity assumptions we refer, e.g., to \cite{AkyildizSabanis2021, BertsekasTsitsiklis2000, DereichKassing2021, DereichMuller_Gronbach2019, FehrmanGessJentzen2020, KarimiMiasojedowMoulinesWai2019arXiv, LeiHuLiTang2020, LovasSabanis2020} and the references mentioned therein.
We also refer, e.g., to \cite{CheriditoJentzenRossmannek2020, JentzenvonWurstemberger2020, LuShinSuKarniadakis2020, Shamir2019} and the references mentioned therein for lower bounds and divergence results for SGD optimization methods.
 For more detailed overviews and further references on SGD optimization schemes we refer, e.g., to \cite{Bottou2018optimization}, \cite[Section 1.1]{FehrmanGessJentzen2020}, \cite[Section 1]{JentzenKuckuckNeufeldVonWurstemberger2018arxiv}, and \cite{Ruder2017overview}.
The effect of random initializations in the training of ANNs was studied, e.g., in \cite{BeckJentzenKuckuck2019arXiv, Hanin2018, HaninRolnick2018, JentzenWelti2020arxiv, LuShinSuKarniadakis2020, ShinKarniadakis2020} and the references mentioned therein.
Another promising branch of research has investigated the convergence of SGD for the training of ANNs in the so-called overparametrized regime, where the number of ANN parameters has to be sufficiently large. In this situation SGD can be shown to converge to global minima with high probability, see, e.g., \cite{ChizatBach2018, DuZhaiPoczosSingh2018arXiv, EMaWu2020, JentzenKroeger2021, LiLiang2019, EChaoWu2018} for the case of shallow ANNs and see, e.g., \cite{AllenzhuLiLiang2019, AllenzhuLiSong2019, DuLeeLiWangZhai2019,  SankararamanDeXuHuangGoldstein2020, ZouCaoZhouGu2019} for the case of deep ANNs. These works consider the empirical risk, which is measured with respect to a finite set of data. 

Another direction of research is to study the true risk landscape of ANNs and characterize the saddle points and local minima, which was done in Cheridito et al.~\cite{Cheridito2021landscape} for the case of affine target functions. The question under which conditions gradient-based optimization algorithms cannot converge to saddle points was investigated, e.g., in \cite{LeePanageasRecht2019, LeeJordanRecht2016, PanageasPiliouras2017, Panageas2019firstorder} for the case of deterministic GD optimization schemes and, e.g., in \cite{Ge2015} for the case of SGD optimization schemes.

In this work we study the plain vanilla SGD optimization method in the training of 
fully-connected feedforward ANNs with ReLU activation in the special situation 
where the target function is a constant function. 
The main result of this work, \cref{theorem:sgd} in \cref{subsection:sgd:convergence}, 
proves that the risk of the SGD process converges to zero in the almost sure 
and the $ L^1 $-sense if the learning rates are sufficiently small but fail to be summable. 
We thereby extend the findings in our previous article Cheridito et al.~\cite{CheriditoJentzenRiekert2021} 
by proving convergence for the SGD optimization method instead of merely for the deterministic GD optimization method, by allowing the gradient to be defined as the limit of the gradients of appropriate general approximations of the ReLU activation function instead of a specific choice for the approximating sequence,
by allowing the learning rates to be non-constant and varying over time, 
by allowing the input data to multi-dimensional,
and by allowing the law of the input data 
to be an arbitrary probability distribution on $[a,b]^d$ with $a \in \R$, $b \in (a, \infty)$, $d \in \N$
instead of the continuous uniform distribution on $ [0,1] $.

To illustrate the findings of this work in more details, we present 
in \cref{theo:intro} below a special case of \cref{theorem:sgd}.
Before we present below the rigorous mathematical statement of \cref{theo:intro}, 
we now provide an informal description of the statement of \cref{theo:intro} 
and also briefly explain some of the mathematical objects 
that appear in \cref{theo:intro} below. 

In \cref{theo:intro} we study the SGD optimization method in the training of fully-connected 
feedforward artificial neural networks (ANNs) with three layers: the input layer, one hidden 
layer, and the output layer. The input layer consists of $ d \in \N = \{ 1, 2, ... \} $ 
neurons (the input is thus $d$-dimensional), the hidden layer consists of  $\width \in \N$
neurons (the hidden layer is thus $\width$-dimensional), and the output layer consists of 1 neuron (the output is thus one-dimensional). 
In between the $ d $-dimensional input layer and the $ \width $-dimensional hidden layer 
an affine linear transformation from $ \R^d $ to $ \R^{ \width } $ 
is applied with $ \width d + \width $ real parameters 
and in between the $ \width $-dimensional hidden layer and the $ 1 $-dimensional output 
layer an affine linear transformation from $ \R^{ \width } $ to $ \R^{ 1 } $ 
is applied with $ \width + 1 $ real parameters. 
Overall the considered ANNs are thus described through
\begin{equation} 
  \fd = ( \width d + \width ) + ( \width + 1 ) = \width d + 2 \width + 1 
\end{equation}
real parameters. In \cref{theo:intro} we assume that the target function 
which we intend to learn is a constant 
and the real number $ \xi \in \R $ in \cref{theo:intro} specifies this constant. 
The real numbers $ a \in \R $, $ b \in (a,\infty) $ in \cref{theo:intro} specify 
the set in which the input data for the training process lies 
in the sense that we assume that the input data is given through $ [a,b]^d $-valued 
i.i.d.\ random variables. 

In \cref{theo:intro} we study the SGD optimization method in the training of ANNs with the rectifier function $\R \ni x \mapsto \max\{ x, 0 \} \in \R$ as the activation function. This type of activation is often also referred to as rectified linear unit activation (ReLU activation). The ReLU activation function $\R \ni x \mapsto \max\{ x, 0 \} \in \R$ fails to be differentiable at the origin and the ReLU activation function can in general therefore not be used to define gradients of the considered risk function and the corresponding gradient descent process. 
In implementations, maybe the most common procedure to overcome this issue 
is to formally apply the chain rule as if all involved functions would be differentiable and to define the ``derivative'' of the ReLU activation function as the left derivative of the ReLU activation function. This is also precisely the way how SGD is implemented in {\sc TensorFlow} and we refer to \cref{subsection:python:code} for a short illustrative example {\sc Python} code for the computation of such generalized gradients of the risk function.

In this article we mathematically formalize this procedure (see \cref{eq:intro:relu:approx}, \cref{setting:sgd:eq:relu}, and \cref{emp:loss:differentiable:item1} in \cref{emp:loss:differentiable}) by employing appropriate continuously differentiable functions which approximate the ReLU activation function in the sense that the employed approximating functions converge to the ReLU activation function and that the derivatives of the employed approximating functions converge to the left derivative of the ReLU activation function. 
More specifically, in \cref{theo:intro} the function $\Rect_{ \infty } \colon \R \to \R$ is the ReLU activation function and the functions $\Rect_r \colon  \R \to \R,$ $r \in \N$, serve as continuously differentiable approximations for the ReLU activation function $\Rect_{ \infty }$. 
In particular, in \cref{theo:intro} we assume that for all $x \in \R$ it holds that 
$ \Rect_{ \infty }( x ) = \max \{ x , 0 \} $ and 
\begin{equation} \label{eq:intro:relu:approx}
 \limsup\nolimits_{r \to \infty}  \abs { \Rect _r ( x ) - \max\{ x, 0 \} } = \limsup\nolimits_{r \to \infty} \abs { (\Rect _r)' ( x ) - \indicator{(0, \infty)} ( x ) }  = 0.
\end{equation}
In \cref{theo:intro} the realization functions associated to the considered ANNs are described through the functions $\scrN_r = ( \realapprox{\phi}{r} )_{\phi \in \R^\fd} \colon \R^\fd \to C ( \R^d , \R)$, $r \in \N \cup \{ \infty \}$. In particular, in \cref{theo:intro} we assume that for all $\phi = ( \phi_1, \ldots, \phi_\fd ) \in \R^\fd$, $x = ( x_1, \ldots, x_d ) \in \R^d$ we have that
\begin{equation}
\realapprox{\phi}{\infty} (x) = \phi_\fd + \smallsum_{i = 1}^\width \phi_{\width ( d+1 )  + i} \max \cu[\big]{ \phi_{ \width d + i} + \smallsum_{j = 1}^d \phi_{ (i - 1 ) d + j} x_j , 0 }
\end{equation}
(cf.\ \cref{theo:intro:eq:realization} below). 
The input data which is used to train the considered ANNs is provided 
through the random variables $X^{ n, m } \colon \Omega \to [a,b]^d$, $n, m \in \N_0$, 
which are assumed to be i.i.d.\ random variables. 
Here $(\Omega, \cF, \P)$ is the underlying probability space.

The function $\cL \colon \R^\fd \to \R$ in \cref{theo:intro} specifies 
the risk function associated to the considered supervised learning problem 
and, roughly speaking, for every neural network parameter $\phi \in \R^{ \fd }$ we have that the value $\cL ( \phi ) \in [0,\infty)$ of the risk function measures the error how well the realization function $\realapprox{\phi}{\infty} \colon \R^d \to \R$ of the neural network associated to $\phi$ approximates the target function $[a,b]^d \ni x \mapsto \xi \in \R$.

The sequence of natural numbers $ ( M_n )_{ n \in \N_0 } \subseteq \N $ describes the size of the mini-batches in the SGD process. 
The SGD optimization method is described through the SGD process $\Theta \colon \N_0 \times \Omega \to \R^\fd$ in \cref{theo:intro} and the real numbers $\gamma_n \in [0, \infty)$, $n \in \N_0$, specify the learning rates in the SGD process. The learning rates are assumed to be sufficiently small in the sense that 
\begin{equation}
  \sup\nolimits_{n \in \N_0} 
  \gamma_n 
  \leq 
  \rbr{ 5 + 5 \norm{\Theta_0 } }^{ - 2 }
  ( \max \{\abs{\xi}, \abs{a}, \abs{b}, d \} )^{ - 5 } 
\end{equation}
and the learning rates may not be summable and instead are assumed to satisfy $\sum_{k=0}^{ \infty } \gamma_k = \infty$. Under these assumptions \cref{theo:intro} proves that the true risk $\cL( \Theta_n )$ 
converges to zero in the almost sure and the $ L^1 $-sense 
as the number of gradient descent steps $ n \in \N $ increases to infinity. 
We now present \cref{theo:intro} and 
thereby precisely formalize the above mentioned paraphrasing comments.

\begin{samepage}
\begin{theorem} \label{theo:intro}
Let $d, \width, \fd \in \N$, $\xi , a  \in \R$, $b \in (a, \infty)$ satisfy $\fd = d\width + 2 \width + 1$,
let $\Rect _r \colon \R \to \R$, $r \in \N \cup \{ \infty\}$, satisfy for all $x \in \R$ that $ \rbr*{ \bigcup_{r \in \N } \{ \Rect _r \}  } \subseteq C^1 ( \R , \R)$, $\Rect _\infty ( x ) = \max \{ x , 0 \}$, and $\limsup_{r \to \infty}  \rbr*{ \abs { \Rect _r ( x ) - \Rect _\infty ( x ) } + \abs { (\Rect _r)' ( x ) - \indicator{(0, \infty)} ( x ) } } = 0$,
let $\scrN_ r = (\realapprox{\phi}{r})_{\phi \in \R^{\fd } } \colon \R^{\fd } \to C(\R ^d , \R)$, $r \in \N \cup \{ \infty \}$, 
satisfy for all $r \in \N \cup \{ \infty \}$, $\phi = ( \phi_1, \ldots, \phi_\fd) \in \R^{\fd}$, $x = (x_1, \ldots, x_d) \in \R^d$ that
\begin{equation} \label{theo:intro:eq:realization}
\realapprox{\phi}{r} (x) = \phi_\fd + \smallsum_{i = 1}^\width \phi_{\width ( d+1 )  + i} \Rect _r \rbr[\big]{ \phi_{ \width d + i} + \smallsum_{j = 1}^d \phi_{ (i - 1 ) d + j} x_j } ,
\end{equation}
let $(\Omega , \cF , \P)$ be a probability space,
let $X^{n , m}  \colon \Omega \to [a,b]^d$, $n, m \in \N_0$, be i.i.d.\ random variables,
let $\norm{ \cdot } \colon \R^\fd \to \R$ and $\cL  \colon \R^\fd \to \R$ satisfy for all $\phi =(\phi_1, \ldots, \phi_{\fd} ) \in \R^\fd$ that $\norm{\phi} = [ \sum_{i=1}^\fd \abs*{ \phi_i } ^2 ] ^{1/2}$ and $\cL  ( \phi) =  \E \br[\big]{( \realapprox{\phi}{\infty} ( X^{0,0}) - \xi  ) ^2 }$,
let $(M_n)_{n \in \N_0} \subseteq \N$,
let $\fL^n_r \colon  \R^{\fd} \times \Omega  \to \R$, $n \in \N_0$, $r \in \N \cup \{ \infty \}$, satisfy for all
$n \in \N_0$, $r \in \N \cup \{ \infty \}$, $\phi \in \R^{\fd}$, $\omega \in \Omega$ that 
\begin{equation} 
\fL^n_r ( \phi , \omega ) = \tfrac{1}{M_n} \smallsum_{m=1}^{M_n}( \realapprox{\phi}{r} ( X^{n,m} ( \omega )) - \xi  ) ^2,
\end{equation}
let $\fG^n  \colon \R^{\fd} \times \Omega \to \R^{\fd}$, $n \in \N_0$, satisfy for all
$n \in \N_0$, $\phi \in \R^\fd$, $ \omega \in  \{ \mathscr{w} \in \Omega \colon  ((\nabla_\phi \fL^n_r ) ( \phi , \mathscr{w} ) )_{r \in \N } \text{ is} \allowbreak\text{convergent} \}$ that $\fG^n ( \phi , \omega ) = \lim_{r \to \infty} (\nabla_\phi \fL^n_r ) ( \phi, \omega )$,
let $\Theta = (\Theta_n)_{n \in \N_0} \colon \N_0 \times \Omega \to \R^{ \fd }$
be a stochastic process, 
let $(\gamma_n)_{n \in \N_0} \subseteq [0, \infty)$,
assume that $\Theta_0$ and $( X^{n,m} )_{(n,m) \in ( \N_0 ) ^2}$ are independent,
and
assume for all $n \in \N_0$, $\omega \in \Omega$ that
$\Theta_{n+1} ( \omega) = \Theta_n (\omega) - \gamma_n \fG ^{n} ( \Theta_n (\omega), \omega)$,
$18 ( \max \{\abs{\xi}, \abs{a}, \abs{b}, d \} ) ^5 \gamma_n \leq \rbr{ 1 + \norm{\Theta_0 ( \omega ) } } ^{-2}$, and $\sum_{k = 0}^\infty \gamma_k = \infty$.
Then
\begin{enumerate} [label=(\roman*)]
\item \label{theo:intro:item1} there exists $\fC \in \R$ such that $\P  \rbr*{  \sup_{n \in \N_0}  \norm{ \Theta_n } \leq \fC  } = 1$,
    \item \label{theo:intro:item2} it holds that $\P  \rbr*{  \limsup_{n \to \infty} \cL  ( \Theta_n ) = 0  } = 1$, and
    \item \label{theo:intro:item3} it holds that $\limsup_{n \to \infty} \E [ \cL  ( \Theta_n ) ] = 0$.
\end{enumerate}
\end{theorem}
\end{samepage}

\cref{theo:intro} is a direct consequence of \cref{cor:sgd:norm} 
in \cref{subsection:sgd:convergence} below. 
\cref{cor:sgd:norm}, in turn, follows from \cref{theorem:sgd}. 
\cref{theorem:sgd} proves that the true risk of the considered SGD processes $(\Theta_n)_{n \in \N_0}$ converges to zero both in the almost sure and the $L^1$-sense in the special case where the target function is constant.
In \cref{section:gd} we establish an analogous result for the deterministic GD optimization method. More specifically, \cref{theo:gd:loss} demonstrates the the true risk of the considered GD processes converges to zero if the target function is constant.
Our proofs of \cref{theo:gd:loss} and \cref{theorem:sgd} make use of similar Lyapunov estimates as in Cheridito et al.~\cite{CheriditoJentzenRiekert2021}. The contradiction argument we use to deal with the case of non-constant learning rates in the proofs of \cref{theo:gd:loss} and \cref{theorem:sgd} is strongly inspired by the arguments in Lei et al.~\cite[Section IV.A]{LeiHuLiTang2020}.

\section{Convergence of gradient descent (GD) processes}
\label{section:gd}

In this section we establish in \cref{theo:gd:loss} in \cref{subsection:gd:convergence} below 
that the true risks of GD processes 
converge 
in the training of ANNs with ReLU activation 
to zero if the target function under consideration 
is a constant. 
\cref{theo:gd:loss} imposes the mathematical framework in \cref{setting:snn} in \cref{subsection:setting:gd} below 
and in \cref{setting:snn} we formally introduce, 
among other things, the considered target function 
$ f \colon [a,b]^d \to \R $ (which is assumed to be an element of the continuous functions 
$ C( [a,b ] ^d, \R ) $ from $ [a,b]^d $ to $ \R $), 
the realization functions 
$ \realapprox{\phi}{\infty} \colon \R^d \to \R $, 
$ \phi \in \R^\fd$, 
of the considered ANNs (see \cref{setting:snn:eq:realization} in \cref{setting:snn}), 
the true risk function $ \cL_{ \infty } \colon \R^\fd \to \R $, 
a sequence of smooth approximations 
$ \Rect_r \colon \R \to \R $, $ r \in \N $, 
of the ReLU activation function 
(see \cref{setting:assumption:rect} in \cref{setting:snn}),
as well as 
the appropriately generalized gradient function 
$ \mathcal{G} = ( \cG_1, \ldots, \cG_\fd) \colon \R^\fd \to \R^\fd $ 
associated to the true risk function. 
In the elementary result in \cref{prop:relu:approx} in \cref{subsection:relu:approx} below 
we also explicitly specify a simple example 
for the considered sequence of smooth approximations 
of the ReLU activation function. 
\cref{prop:relu:approx} is, e.g., proved as 
Cheridito et al.~\cite[Proposition 2.2]{CheriditoJentzenRiekert2021}.

\Cref{theo:gd:item2} in \cref{theo:gd:loss} in \cref{subsection:gd:convergence} below 
shows that the 
true risk 
$ \mathcal{L}_{ \infty }( \Theta_n ) $
of the GD process $ \Theta \colon \N_0 \to \R^{ \fd } $ 
converges to zero as the number 
of gradient descent steps $ n \in \N $ increases to infinity. 
In our proof of \cref{theo:gd:loss} we use 
the upper estimates for the standard norm of the generalized gradient function 
$ \mathcal{G} \colon \R^{ \mathfrak{d} } \to \R^{ \mathfrak{d} } $ 
in \cref{lem:gradient:est} and \cref{cor:g:bounded} in \cref{subsection:gradient:est} below 
as well as the Lyapunov type estimates for GD processes 
in \cref{lem:gradient:descent:gen}, \cref{cor:gradient:descent:gen}, \cref{cor:est:vtheta_n}, and \cref{lem:vthetan:decreasing} in 
\cref{subsection:lyapunov:gd} below. 
Our proof of 
\cref{cor:g:bounded}
employs \cref{lem:gradient:est} and the elementary 
local Lipschitz continuity estimates 
for the true risk function 
in \cref{lem:realization:lip} below. 
\cref{lem:realization:lip} is a direct consequence of, e.g., Beck et al.~\cite[Theorem 2.36]{BeckJentzenKuckuck2019arXiv}. 
Our proof of \cref{lem:gradient:est} makes use of the 
elementary representation result for 
the generalized gradient function $ \cG \colon \R^\fd \to \R^\fd $ 
in \cref{prop:limit:lr} in \cref{subsection:approx:gradient} below.

Our proof of \cref{cor:gradient:descent:gen} employs 
\cref{lem:gradient:est} and the elementary lower and upper estimates 
for the Lyapunov function 
$ \R^{ \mathfrak{d} } \ni \phi \mapsto \norm{\phi}^2 + \abs{\phi_\fd - 2 f(0) }^2 \in \R $
in \cref{prop:lyapunov:norm} below. 
Our proof of \cref{lem:gradient:descent:gen} uses the elementary representation result 
for the gradient function 
of the Lyapunov function 
$ V \colon \R^\fd \to \R $ 
in \cref{prop:v:gradient} in \cref{subsection:lyapunov:elementary} below 
as well as the identities 
for the gradient flow dynamics 
of the Lyapunov function in 
\cref{prop:lyapunov:gradient} and \cref{cor:lyapunov:const} in \cref{subsection:lyapunov:elementary} below.

The findings in this section extend and/or generalize 
the findings in 
Section 2 and Section 3 
in Cheridito et al.~\cite{CheriditoJentzenRiekert2021} (to the more general and multi-dimensional 
setup considered in \cref{setting:snn} in \cref{subsection:setting:gd}). In particular, 
\cref{prop:relu:approx} in \cref{subsection:relu:approx} is a direct consequence of \cite[Proposition 2.2]{CheriditoJentzenRiekert2021},
\cref{prop:limit:lr} in \cref{subsection:approx:gradient} extends \cite[Proposition 2.3]{CheriditoJentzenRiekert2021}, 
\cref{lem:gradient:est} in \cref{subsection:gradient:est} generalizes \cite[Lemma 2.10]{CheriditoJentzenRiekert2021},
\cref{cor:g:bounded} in \cref{subsection:gradient:est} generalizes \cite[Corollary 2.11]{CheriditoJentzenRiekert2021},
\cref{prop:lyapunov:norm} in \cref{subsection:lyapunov:elementary} generalizes \cite[Proposition 2.12]{CheriditoJentzenRiekert2021},
\cref{prop:v:gradient} in \cref{subsection:lyapunov:elementary} generalizes Proposition \cite[Proposition 2.13]{CheriditoJentzenRiekert2021},
\cref{cor:lyapunov:const} in \cref{subsection:lyapunov:elementary} extends \cite[Proposition 2.14]{CheriditoJentzenRiekert2021}, 
\cref{cor:critical:points} in \cref{subsection:lyapunov:elementary} generalizes \cite[Corollary 2.15]{CheriditoJentzenRiekert2021},
\cref{lem:gradient:descent:gen} in \cref{subsection:lyapunov:gd} extends \cite[Lemma 4.1]{CheriditoJentzenRiekert2021},
\cref{cor:est:vtheta_n} in \cref{subsection:lyapunov:gd} generalizes \cite[Corollary 4.2]{CheriditoJentzenRiekert2021},
\cref{lem:vthetan:decreasing} in \cref{subsection:lyapunov:gd} extends \cite[Lemma 4.3]{CheriditoJentzenRiekert2021}, and
\cref{theo:gd:loss} in \cref{subsection:lyapunov:gd} generalizes \cite[Theorem 4.4]{CheriditoJentzenRiekert2021}. 

\subsection{Description of artificial neural networks (ANNs) with ReLU activation}
\label{subsection:setting:gd}

\begin{setting} \label{setting:snn} 
Let $d, \width, \fd \in \N$, $ \bfa , a \in \R$, $b \in (a, \infty)$, $f \in C ( [a , b]^d , \R)$ satisfy $\fd = d\width + 2 \width + 1$ and $\bfa = \max \{ \abs{a}, \abs{b} , 1 \}$,
let $\fw  = (( \w{\phi} _ {i,j}  )_{(i,j) \in \{1, \ldots, \width \} \times \{1, \ldots, d \} })_{ \phi \in \R^{\fd}} \colon \R^{\fd} \to \R^{ \width \times d}$,
$\fb =  (( \b{\phi} _ 1 , \ldots, \b{\phi} _ \width ))_{ \phi \in \R^{\fd}} \colon \R^{\fd} \to \R^{\width}$,
$\fv = (( \v{\phi} _ 1 , \ldots, \v{\phi} _ \width ))_{ \phi \in \R^{\fd}} \colon \R^{\fd} \to \R^{\width}$, and
$\fc = (\c{\phi})_{\phi \in \R^{\fd }} \colon \R^{\fd} \to \R$
 satisfy for all $\phi  = ( \phi_1 ,  \ldots, \phi_{\fd}) \in \R^{\fd}$, $i \in \{1, 2, \ldots, \width \}$, $j \in \{1, 2, \ldots, d \}$ that $\w{\phi}_{i , j} = \phi_{ (i - 1 ) d + j}$, $\b{\phi}_i = \phi_{\width d + i}$, 
$\v{\phi}_i = \phi_{ \width ( d+1 )  + i}$, and $\c{\phi} = \phi_{\fd}$,
let $\Rect _r \colon \R \to \R$, $r \in \N \cup \{ \infty\}$, satisfy for all  $x \in \R$ that $ \rbr*{ \bigcup_{r \in \N } \{ \Rect _r \}  } \subseteq C^1 ( \R , \R)$, $\Rect _\infty ( x ) = \max \{ x , 0 \}$, $\sup_{r \in \N} \sup_{y \in [-\abs{x}  , \abs{ x } ]}  \abs{ (\Rect _r)'(y) }  < \infty$, and 
\begin{equation} \label{setting:assumption:rect}
     \limsup\nolimits_{r \to \infty} \rbr*{ \abs { \Rect _r ( x ) - \Rect _\infty ( x ) } + \abs { (\Rect _r)' ( x ) - \indicator{(0, \infty)} ( x ) } } = 0,
\end{equation}
let $\mu \colon \cB ( [ a,b] ^d ) \to [0,1]$ be a probability measure,
let $\scrN_ r = (\realapprox{\phi}{r})_{\phi \in \R^{\fd } } \colon \R^{\fd } \to C(\R^d , \R)$, $r \in \N \cup \{ \infty \}$, and $\cL_r \colon \R^{\fd  } \to \R$, $r \in \N \cup \{ \infty \}$,
satisfy for all $r \in \N \cup \{ \infty \}$, $\phi \in \R^{\fd}$, $x = (x_1, \ldots, x_d) \in \R^d$ that 
\begin{equation} \label{setting:snn:eq:realization}
\realapprox{\phi}{r} (x) = \c{\phi} + \smallsum_{i = 1}^\width \v{\phi}_i \Rect _r \rbr[\big]{ \b{\phi}_i + \smallsum_{j=1}^d \w{\phi}_{i,j} x_j }
\end{equation}
and $\cL_r(\phi) = \int_{[a ,b]^d} (\realapprox{\phi}{r} (y) - f ( y ) )^2 \, \mu ( \d y )$,
let $\cG = (\cG_1, \ldots, \cG_{\fd}) \colon \R^{\fd} \to \R^{\fd}$ satisfy for all
$\phi \in  \{ \varphi \in \R^{\fd} \colon   ((\nabla \cL_r ) ( \varphi ) )_{r \in \N } \text{ is convergent} \}$ that $\cG ( \phi ) = \lim_{r \to \infty} (\nabla \cL_r ) ( \phi )$,
let $\norm{ \cdot } \colon \rbr*{  \bigcup_{n \in \N} \R^n  } \to \R$ and $\langle \cdot , \cdot \rangle \colon \rbr*{  \bigcup_{n \in \N} (\R^n \times \R^n )  } \to \R$ satisfy for all $n \in \N$, $x=(x_1, \ldots, x_n)$, $y=(y_1, \ldots, y_n ) \in \R^n $ that $\norm{ x } = [ \sum_{i=1}^n \abs*{ x_i } ^2 ] ^{1/2}$ and $\langle x , y \rangle = \sum_{i=1}^n x_i y_i$,
and let $I_i^\phi \subseteq \R^d$, $\phi \in \R^{\fd }$, $i \in \{1, 2, \ldots, \width \}$, and $V \colon \R^{\fd } \to \R$ satisfy for all 
$\phi \in \R^{\fd}$, $i \in \{1, 2, \ldots, \width \}$ that $I_i^\phi = \{ x = (x_1, \ldots, x_d) \in [a,b]^d \colon \b{\phi}_i + \sum_{j = 1}^d \w{\phi}_{i,j} x_j > 0 \}$ and $V(\phi) = \norm{ \phi } ^2 + \abs{ \c{\phi} -  2 f ( 0 ) } ^2$.
\end{setting}

\subsection{Smooth approximations for the ReLU activation function}
\label{subsection:relu:approx}

\begin{prop} \label{prop:relu:approx}
Let $\Rect_r \colon \R \to \R$, $r \in \N$, satisfy for all $r \in \N$, $x \in \R$ that $\Rect _r ( x ) = r^{-1} \ln ( 1 + r^{-1} e^{r x } )$. Then
\begin{enumerate} [label=(\roman*)] 
    \item it holds for all $r \in \N$ that $\Rect_r \in C^\infty ( \R , \R)$,
    \item it holds for all $x \in \R$ that $\limsup_{r \to \infty}  \abs { \Rect _r ( x ) - \max \{ x , 0 \} }= 0$,
    \item it holds for all $x \in \R$ that $\limsup_{r \to \infty}  \abs { (\Rect _r)' ( x ) - \indicator{(0, \infty)} ( x ) }  = 0$, and
    \item it holds that $\sup_{r \in \N} \sup_{x \in \R}  \abs{ (\Rect _r)' (x) }  \leq 1 < \infty$.
\end{enumerate}
\end{prop}

\subsection{Properties of the approximating true risk functions and their gradients}
\label{subsection:approx:gradient}

\begin{prop} \label{prop:limit:lr}
Assume \cref{setting:snn} and let $\phi  = (\phi_1, \ldots, \phi_{\fd}) \in \R^{\fd}$. Then 
\begin{enumerate} [label=(\roman*)]
\item \label{prop:limit:lr:1} it holds for all $r \in \N$ that $\cL _ r \in C^1 ( \R^{\fd }, \R)$,
    \item \label{prop:limit:lr:2} it holds for all $r \in \N$, $i \in \{1, 2, \ldots, \width \}$, $j \in \{1, 2, \ldots, d \}$ that
\begin{equation} \label{eq:approx:loss:gradient}
    \begin{split}
        \rbr[\big]{  \tfrac{\partial  }{ \partial \phi_{ (i - 1 ) d + j}} \cL_r } ( \phi ) &= 2 \v{\phi}_i \int_{[a,b]^d} x_j \br[\big]{  (\Rect _r )' \rbr[\big]{ \b{\phi}_i + \smallsum_{k = 1}^d \w{\phi}_{i, k} x_k }} ( \realapprox{\phi}{r}(x) - f ( x ) ) \, \mu ( \d x ), \\
         \rbr[\big]{  \tfrac{\partial }{ \partial \phi_{\width d + i}}  \cL_r } ( \phi ) &= 2 \v{\phi}_i \int_{[a,b]^d}  \br[\big]{  (\Rect _r )' \rbr[\big]{ \b{\phi}_i + \smallsum_{k = 1}^d \w{\phi}_{i, k} x_k } } ( \realapprox{\phi}{r}(x) - f ( x ) ) \, \mu ( \d x ), \\
          \rbr[\big]{  \tfrac{\partial }{ \partial \phi_{ \width ( d+1 )  + i }}  \cL_r } ( \phi ) &= 2  \int_{[a,b]^d} \br[\big]{  \Rect _r\rbr[\big]{ \b{\phi}_i + \smallsum_{k = 1}^d \w{\phi}_{i, k} x_k } } ( \realapprox{\phi}{r}(x) - f ( x ) ) \, \mu ( \d x ),  \\
          \text{and} \qquad \rbr[\big]{  \tfrac{\partial  }{ \partial \phi_\fd} \cL_r } ( \phi ) &= 2  \int_{[a,b]^d} ( \realapprox{\phi}{r}(x) - f (x )) \, \mu(\d x),
    \end{split}
\end{equation}
\item \label{prop:limit:lr:3} it holds that $\limsup_{r \to \infty} \abs{ \cL_r ( \phi ) - \cL_\infty ( \phi) } = 0$,
\item \label{prop:limit:lr:4} it holds that $\limsup_{r \to \infty } \norm{ ( \nabla \cL_ r ) ( \phi ) - \cG ( \phi ) }  = 0$, and
\item \label{prop:limit:lr:5} it holds for all $i \in \{1, 2, \ldots, \width \}$, $j \in \{1, 2, \ldots, d \}$ that
\begin{equation} \label{eq:loss:gradient}
\begin{split}
        \cG_{ (i - 1 ) d + j} ( \phi) &= 2 \v{\phi}_i \int_{I_i^\phi} x _j ( \realapprox{\phi}{\infty} (x) - f ( x ) ) \, \mu ( \d x ), \\
        \cG_{ \width d + i} ( \phi) &= 2 \v{\phi}_i \int_{I_i^\phi} (\realapprox{\phi}{\infty} (x) - f ( x ) ) \, \mu ( \d x ), \\
        \cG_{\width ( d+1 )  + i} ( \phi) &= 2 \int_{[a,b]^d} \br[\big]{\Rect _\infty \rbr[\big]{\b{\phi}_i + \smallsum_{k = 1}^d \w{\phi}_{i, k} x_k } } ( \realapprox{\phi}{\infty}(x) - f ( x ) ) \, \mu ( \d x ), \\
        \text{and} \qquad \cG_{\fd} ( \phi) &= 2 \int_{[a,b]^d} (\realapprox{\phi}{\infty} (x) - f ( x ) ) \, \mu ( \d x ).
        \end{split}
\end{equation}

\end{enumerate}
\end{prop}
\begin{cproof} {prop:limit:lr}
Throughout this proof let $\fM \colon [0, \infty) \to [0, \infty]$ satisfy for all $x \in [0, \infty)$ that $\fM ( x ) = \sup_{r \in \N} \sup_{y \in [-x,x]} \rbr*{ \abs{\Rect _r ( y ) } + \abs{(\Rect _r)' ( y ) } }$.
\Nobs that the assumption that for all $r \in \N$ it holds that $\Rect _r \in C^1 ( \R , \R)$ implies that for all $r \in \N$, $x \in \R$ we have that $\Rect _r(x) = \Rect _r(0) + \int_0^x (\Rect _r)'(y) \, \d y$. This, the assumption that for all $x \in \R$ it holds that $\sup_{r \in \N} \sup_{y \in [-\abs{x}  , \abs{ x } ]}  \abs{ (\Rect _r)'(y) }  < \infty$ and the fact that $\sup_{r \in \N} \abs{\Rect _r(0)} < \infty$ prove that for all $x \in [0, \infty)$ it holds that $\sup_{r \in \N} \sup_{y \in [-x,x]} \abs{\Rect _r ( y ) }  < \infty$. Hence, we obtain that for all $x \in [0, \infty)$ it holds that $\fM ( x ) < \infty$.
This, the assumption that for all $r \in \N$ it holds that $\Rect _r \in C^1 ( \R , \R)$, the chain rule, and the dominated convergence theorem establish \cref{prop:limit:lr:1,prop:limit:lr:2}.
Next \nobs that for all $r \in \N$, $x = (x_1, \ldots, x_d) \in [a,b]^d$ it holds that
\begin{equation} \label{proof:limit:lr:eq1}
\begin{split}
     \abs{ \realapprox{\phi}{r} ( x ) - f ( x ) }
   &\leq \br[\big]{ \sup\nolimits_{y \in [a,b]^d } \abs { f(y)} } + \abs{ \c{\phi} } + \smallsum_{i = 1}^\width \abs{\v{\phi}_i}  \br[\big]{\Rect _r \rbr[\big]{ \b{\phi}_i + \smallsum_{j=1}^d \w{\phi}_{i,j} x_j } } \\
   &\leq \br[\big]{ \sup\nolimits_{y \in [a,b]^d } \abs { f(y)} } + \abs{ \c{\phi} } + \smallsum_{i = 1}^\width \abs{\v{\phi}_i} \br[\big]{ \fM \rbr[\big]{ \abs{\b{\phi}_i} + \bfa \smallsum_{j = 1}^d \abs{ \w{\phi}_{i,j}} } } .
   \end{split}
\end{equation}
The fact that for all $x \in [a,b]^d$ it holds that $\lim_{r \to \infty} ( \realapprox{\phi}{r} ( x ) - f ( x ) ) = \realapprox{\phi}{\infty} ( x ) - f ( x ) $ and the dominated convergence theorem hence prove that $\lim_{r \to \infty} \cL_r ( \phi ) = \cL_\infty ( \phi)$. This establishes \cref{prop:limit:lr:3}.
Moreover, \nobs that \cref{proof:limit:lr:eq1}, the dominated convergence theorem, and the fact that for all $x \in [a,b]^d$ it holds that $ \lim_{r \to \infty} ( \realapprox{\phi}{r} ( x ) - f (x ) ) = \realapprox{\phi}{\infty} ( x ) - f ( x )$ assure that
\begin{equation} \label{proof:limit:lr:eq2}
    \lim_{r \to \infty}  \br[\big]{  \rbr[\big]{ \tfrac{\partial  }{ \partial \phi_\fd} \cL_r } ( \phi ) }
    = 2 \int_{[a,b]^d} (\realapprox{\phi}{\infty} (x) - f ( x ) ) \, \mu ( \d x ).
\end{equation}
Next \nobs that for all $x =(x_1, \ldots, x_d) \in [a,b]^d$, $i \in \{1, 2, \ldots, \width \}$, $j \in \{1, 2, \dots, d \}$ we have that
\begin{equation} \label{prop:limit:lr:eq:partialw}
    \begin{split}
        &\lim_{r \to \infty}  \br[\big]{ x _j \br[\big]{  (\Rect _r) ' \rbr[\big]{ \b{\phi}_i + \smallsum_{k = 1}^d \w{\phi}_{i, k} x_k } } ( \realapprox{\phi}{r}(x) - f ( x ) ) } \\
        &= x_j ( \realapprox{\phi}{\infty} ( x ) - f ( x ) ) \indicator{(0 , \infty ) }  \rbr[\big]{ \b{\phi}_i + \smallsum_{k = 1}^d \w{\phi}_{i, k} x_k } \\
        &= x _j ( \realapprox{\phi}{\infty} ( x ) - f ( x ) )\indicator{I_i^\phi} ( x )
    \end{split}
\end{equation}
and
\begin{equation} \label{prop:limit:lr:eq:partialb}
\begin{split}
    & \lim_{r \to \infty} \br[\big]{   [(\Rect _r) ' \rbr[\big]{ \b{\phi}_i + \smallsum_{k = 1}^d \w{\phi}_{i, k} x_k }] ( \realapprox{\phi}{r}(x) - f ( x )) } \\
     &=  ( \realapprox{\phi}{\infty} ( x ) - f ( x ) ) \indicator{(0 , \infty ) }\rbr[\big]{ \b{\phi}_i + \smallsum_{k = 1}^d \w{\phi}_{i, k} x_k } \\
     &=  ( \realapprox{\phi}{\infty} ( x ) - f ( x ) )  \indicator{I_i^\phi} ( x ).
     \end{split}
\end{equation}
Furthermore, \nobs that \cref{proof:limit:lr:eq1} shows that for all $r \in \N$, $x = (x_1, \ldots, x_d) \in [a , b ]^d$, $i \in \{1, 2, \ldots, \width \}$, $j \in \{1, 2, \ldots, d \}$, $v \in \{ 0, 1 \}$ it holds that
\begin{equation} \label{prop:limit:lr:eq:partial:bounded}
\begin{split}
    &\abs[\big]{ ( x_j )^v \br[\big]{ (\Rect _r) ' \rbr[\big]{ \b{\phi}_i + \smallsum_{k = 1}^d \w{\phi}_{i, k} x_k } } ( \realapprox{\phi}{r}(x) - f ( x )) } \\
    &\leq \bfa  \br[\Big]{ \abs[\big]{ (\Rect _r ) ' \rbr[\big]{ \b{\phi}_i + \smallsum_{k = 1}^d \w{\phi}_{i, k} x_k} } } \br[\big]{ \abs{ \realapprox{\phi}{r}(x) - f ( x ) } }\\
    & \leq \bfa \br[\big]{ \fM \rbr[\big]{ \abs{\b{\phi}_i} + \bfa \smallsum_{k = 1}^d \abs{\w{\phi}_{i, k}} } } \br[\big]{ \abs{ \realapprox{\phi}{r} ( x ) - f ( x ) } } \\
   &\leq \bfa \br[\big]{ \fM \rbr[\big]{ \abs{\b{\phi}_i} + \bfa \smallsum_{k = 1}^d \abs{\w{\phi}_{i, k} } } } \Bigl( \br[\big]{ \sup\nolimits_{y \in [a,b]^d} \abs{f ( y ) } } \Bigr. \\
   & \Bigl. \quad + \abs{ \c{\phi} } + \smallsum_{\ell = 1}^\width \abs{\v{\phi}_ \ell} \br[\big]{ \fM \rbr[\big]{\abs{\b{\phi}_\ell } + \bfa \smallsum_{m = 1 }^d \abs{ \w{\phi}_{\ell , m}} } }  \Bigr) .
   \end{split}
\end{equation}
The dominated convergence theorem and \cref{prop:limit:lr:eq:partialw} hence prove that for all $i \in \{1, 2, \ldots, \width \}$, $j \in \{1, 2, \ldots, d \}$ we have that
\begin{equation} \label{proof:limit:lr:eq3}
\begin{split}
    \lim_{r \to \infty} \br[\big]{  \rbr[\big]{  \tfrac{\partial  }{ \partial \phi_{ (i - 1 ) d + j}} \cL_r } ( \phi ) }
    &=  2 \v{\phi}_i \int_{[a,b]^d} x_j ( \realapprox{\phi}{\infty} ( x ) - f ( x ) )  \indicator{I_i^\phi} ( x ) \, \mu ( \d x ) \\
    & = 2 \v{\phi}_i \int_{I_i^\phi} x_j ( \realapprox{\phi}{\infty} (x) - f ( x ) ) \, \mu (\d  x ) .
    \end{split}
\end{equation}
Moreover, \nobs that \cref{prop:limit:lr:eq:partialb}, \cref{prop:limit:lr:eq:partial:bounded}, and the dominated convergence theorem demonstrate that for all $i \in \{1, 2, \ldots, \width \}$, $j \in \{1, 2, \ldots, d \}$ it holds that
\begin{equation} \label{proof:limit:lr:eq4}
\begin{split}
    \lim_{r \to \infty} \br[\big]{  \rbr[\big]{  \tfrac{\partial  }{ \partial \phi_{ \width d + i}} \cL_r } ( \phi ) }
    &= 2 \v{\phi}_i \int_{[a,b]^d} ( \realapprox{\phi}{\infty} ( x ) - f ( x ) ) \indicator{I_i^\phi} ( x ) \, \mu ( \d x ) \\
    & =  2 \v{\phi}_i \int_{I_i^\phi } ( \realapprox{\phi}{\infty} (x) - f ( x ) ) \, \mu ( \d x ) .
    \end{split}
\end{equation}
Furthermore, \nobs that for all $x \in [a , b]^d$, $i \in \{1, 2, \ldots, \width \}$ it holds that
\begin{equation} \label{prop:limit:lr:eq:partialv}
    \lim_{r \to \infty} \br[\big]{  \Rect _r \rbr[\big]{ \b{\phi}_i + \smallsum_{j = 1}^d \w{\phi}_{i,j} x_j } } ( \realapprox{\phi}{r}(x) - f ( x ) )
    = \br[\big]{\Rect  _\infty  \rbr[\big]{ \b{\phi}_i + \smallsum_{j = 1}^d \w{\phi}_{i,j} x_j } } ( \realapprox{\phi}{\infty}(x) - f ( x ) ) .
\end{equation}
In addition, \nobs that \cref{proof:limit:lr:eq1} ensures that for all $r \in \N$, $x \in [a,b]^d$, $i \in \{1, 2, \ldots, \width\}$ we have that
\begin{equation}
\begin{split}
    &\abs[\big]{  \br[\big]{  \Rect _r \rbr[\big]{ \b{\phi}_i + \smallsum_{j = 1}^d \w{\phi}_{i,j} x_j } } ( \realapprox{\phi}{r}(x) - f ( x ) )  } \\
    &\leq  \br[\big]{ \fM \rbr[\big]{ \abs{ \b{\phi}_i } + \bfa \smallsum_{j = 1}^d \abs{\w{\phi}_{i,j}} } } \abs{ \realapprox{\phi}{r} ( x ) - f ( x ) } \\
   &\leq \br[\big]{ \fM \rbr[\big]{\abs{ \b{\phi}_i } + \bfa \smallsum_{j = 1}^d \abs{\w{\phi}_{i,j}} } } \Bigl( \br[\big]{ \sup\nolimits_{y \in [a,b]^d} \abs{f(y)} } \Bigr. \\
   & \Bigl. \quad+ \abs{ \c{\phi} } + \smallsum_{\ell = 1}^\width \abs{\v{\phi}_\ell } \br[\big]{ \fM \rbr[\big]{ \abs{\b{\phi}_\ell} + \bfa \smallsum_{m = 1}^d \abs{ \w{\phi}_{\ell , m}} } } \Bigr) .
\end{split}
\end{equation}
This, \cref{prop:limit:lr:eq:partialv}, and the dominated convergence theorem demonstrate that for all $i \in \{1, 2, \ldots, \width \}$ it holds that 
\begin{equation} 
\lim_{r \to \infty} \br[\big]{   \rbr[\big]{  \tfrac{\partial  }{ \partial \phi_{ \width ( d+1 )  + i}} \cL_r } ( \phi ) } = 2 \int_{[a,b]^d} \br[\big]{\Rect  _\infty  \rbr[\big]{\b{\phi}_i + \smallsum_{j = 1}^d \w{\phi}_{i,j} x_j } } ( \realapprox{\phi}{\infty}(x) - f ( x ) ) \, \mu ( \d x ) .
\end{equation}
Combining this, \cref{proof:limit:lr:eq2,proof:limit:lr:eq3,proof:limit:lr:eq4} establishes \cref{prop:limit:lr:4,prop:limit:lr:5}. 
\end{cproof}

\subsection{Local Lipschitz continuity properties of the true risk functions}
\label{subsection:loss:lipschitz}

\cfclear
\begin{lemma} \label{lem:realization:lip}
Let $d, \width, \fd \in \N$, $ a \in \R$, $b \in [ a, \infty)$, $f \in C ( [a , b]^d , \R)$ satisfy $\fd = d\width + 2 \width + 1$,
let $\scrN = (\realization{\phi})_{\phi \in \R^{\fd } } \colon \R^{\fd } \to C(\R ^d , \R)$
satisfy for all $\phi = ( \phi_1, \ldots, \phi_\fd) \in \R^{\fd}$, $x = (x_1, \ldots, x_d) \in \R^d$ that
\begin{equation}
\realization{\phi} (x) = \phi_\fd + \smallsum_{i = 1}^\width \phi_{\width ( d+1 )  + i} \max \cu[\big]{\phi_{ \width d + i} + \smallsum_{j = 1}^d \phi_{ (i - 1 ) d + j} x_j , 0 } ,
\end{equation}
let $\mu \colon \cB ( [a,b]^d ) \to [0,1]$ be a probability measure,
let $\norm{ \cdot } \colon \R^\fd \to \R$ and $\cL \colon \R^\fd \to \R$ satisfy for all $\phi =(\phi_1, \ldots, \phi_{\fd} ) \in \R^\fd$ that $\norm{\phi} = [ \sum_{i=1}^\fd \abs*{ \phi_i } ^2 ] ^{1/2}$ and $\cL (\phi) = \int_{[a ,b]^d} (\realization{\phi} (y) - f ( y ) )^2 \, \mu ( \d y )$,
and let $K \subseteq \R^{\fd }$ be compact. Then there exists $\scrL \in \R$ such that for all $\phi, \psi \in K$ it holds that 
\begin{equation} \label{lem:realization:lip:eq}
   \sup\nolimits_{x \in [a,b]^d} \abs{ \realization{\phi} ( x ) - \realization{ \psi} ( x ) } + \abs{ \cL( \phi ) - \cL ( \psi ) } \leq \scrL \norm{ \phi - \psi }.
\end{equation}
\end{lemma}
\begin{cproof}{lem:realization:lip}
Throughout this proof let $\bfa \in \R$ satisfy $\bfa = \max \{ \abs{a} , \abs{b} , 1\}$.
\Nobs that, e.g., Beck et al.~\cite[Theorem 2.36]{BeckJentzenKuckuck2019arXiv} (applied with $a \with a$, $b \with b$, $d \with \fd$, $L \with 2$, $l_0 \with d$, $l_1 \with \width $, $l_2 \with 1$ in the notation of \cite[Theorem 2.36]{BeckJentzenKuckuck2019arXiv}) and the fact that for all  $\varphi = ( \varphi_1 , \ldots, \varphi_{\fd } ) \in \R^{\fd }$ it holds that $\max_{i \in \{1, 2, \ldots, \fd \}} \abs{ \varphi _ i } \leq \norm{ \varphi }$ demonstrate that for all $\phi, \psi \in \R^{\fd}$ it holds that
\begin{equation} \label{lem:realization:lip:eq1}
    \sup\nolimits_{x \in [a,b]^d} \abs{ \realization{\phi} ( x ) - \realization{ \psi} ( x ) } \leq 2 \bfa (d+1) ( \width + 1 )  ( \max \{1 , \norm{ \phi } , \norm{ \psi } \} ) \norm{ \phi - \psi } .
\end{equation}
Furthermore, \nobs that the fact that $K$ is compact ensures that there exists $\kappa \in [1 , \infty) $ such that for all $ \varphi \in K$ it holds that 
\begin{equation} \label{lem:realization:lip:eq2}
    \norm{ \varphi } \leq \kappa .
\end{equation}
Note that \cref{lem:realization:lip:eq1} and \cref{lem:realization:lip:eq2} show that there exists $\scrL \in \R$ which satisfies for all $\phi, \psi \in K$ that 
\begin{equation} \label{lem:realization:lip:eq3}
\sup\nolimits_{x \in [a,b]^d} \abs{ \realization{\phi} ( x ) - \realization{ \psi} ( x ) } \leq \scrL \norm{ \phi - \psi }.
\end{equation} 
Hence, we obtain that for all $\phi, \psi \in K$ it holds that
\begin{equation} \label{loss:lip:eq1}
    \begin{split}
        \abs{ \cL ( \phi ) - \cL ( \psi ) } &=  \abs*{ \br*{ \int_{[a,b]^d} ( \realization{\phi} (x) - f(x)) ^2 \, \mu ( \d x ) } - \br*{ \int_{[a,b]^d} ( \realization{\psi} (x) - f(x)) ^2 \, \mu ( \d x ) } } \\
        &\leq \int_{[a,b]^d} \abs[\big]{  ( \realization{\phi} (x) - f(x)) ^2 - ( \realization{\psi} (x) - f(x)) ^2 } \, \mu ( \d x ) \\
        &= \int_{[a,b]^d} \abs[\big]{ \realization{\phi} ( x ) - \realization{ \psi} ( x ) }  \abs[\big]{ \realization{\phi} ( x ) + \realization{ \psi} ( x ) - 2 f(x) } \, \mu ( \d x ) \\
        &\leq \scrL \norm{ \phi - \psi } \br*{ \int_{[a,b]^d} \abs[\big]{ \realization{\phi} ( x ) + \realization{ \psi} ( x ) - 2 f (x) } \, \mu (\d x) }.
    \end{split}
\end{equation}
This, \cref{lem:realization:lip:eq2}, \cref{lem:realization:lip:eq3}, and the fact that for all $x \in [a,b]^d$ it holds that $\realization{0} ( x ) = 0$ prove that for all $\phi , \psi \in K$ we have that
\begin{equation}
    \begin{split}
        \abs{\cL ( \phi ) - \cL ( \psi ) } 
       & \leq \scrL \norm{ \phi - \psi } \rbr*{  \sup\nolimits_{x \in [a,b]^d}  \br[\big]{  \abs{ \realization{\phi} ( x ) } + \abs{ \realization{ \psi} ( x ) }  + 2 \abs{f(x) } } } \\
        &=  \scrL \norm{ \phi - \psi } \rbr*{  \sup\nolimits_{x \in [a,b]^d}  \br[\big]{  \abs{ \realization{\phi} ( x ) - \realization{0}(x) } + \abs{ \realization{ \psi} ( x ) -\realization{0} ( x ) }  + 2 \abs{f(x) } } } \\
        &\leq \scrL  \norm{\phi - \psi } \rbr*{\scrL \norm{\phi} + \scrL \norm{\psi} + 2 \br[\big]{ \sup\nolimits_{x \in [a,b]^d }\abs{f( x )} } } \\
        &\leq 2 \scrL \rbr*{ \kappa \scrL +  \br[\big]{ \sup\nolimits_{x \in [a,b]^d }\abs{f( x ) } } } \norm{\phi - \psi}.
    \end{split}
\end{equation}
Combining this with \cref{lem:realization:lip:eq3} establishes \cref{lem:realization:lip:eq}.
\end{cproof}

\subsection{Upper estimates for generalized gradients of the true risk functions}
\label{subsection:gradient:est}

\cfclear
\begin{lemma} \label{lem:gradient:est}
Assume \cref{setting:snn} and let $\phi \in \R^{\fd}$. Then
\begin{equation} 
    \norm{ \cG ( \phi ) } ^2 \leq 4 ( \bfa^2 ( d + 1) \norm{ \phi } ^2 + 1 ) \cL_\infty ( \phi ).
\end{equation}
\end{lemma}
\begin{cproof} {lem:gradient:est}
\Nobs that Jensen's inequality implies that
\begin{equation} \label{lem:grad:est:eq1}
    \rbr*{  \int_{[a,b]^d} \abs{ \realapprox{\phi}{\infty} ( x ) - f ( x ) } \, \mu ( \d x )  } ^{\! \! 2} \leq \int_{[a,b]^d} \rbr{  \realapprox{\phi}{\infty} ( x ) - f ( x )  } ^2 \, \mu ( \d x ) = \cL_\infty ( \phi ).
\end{equation}
Combining this and \cref{eq:loss:gradient} demonstrates that for all $i \in \{1, 2, \ldots, \width \}$, $j \in \{1,2, \ldots, d \}$ it holds that
\begin{equation} \label{eq:lem:gradient:est1}
    \begin{split}
        \abs{ \cG_{ (i - 1 ) d + j }( \phi) } ^2 &= 4 (\v{\phi}_i) ^2 \rbr*{ \int_{I_i^\phi} x_j ( \realapprox{\phi}{\infty} (x) - f ( x )  ) \, \mu ( \d x )  } ^{\! \! 2} \\
        & \leq 4 (\v{\phi}_i) ^2 \rbr*{ \int_{I_i^\phi} \abs{x_j}  \abs{ \realapprox{\phi}{\infty} (x) - f ( x )  } \, \mu ( \d x )  } ^{\! \! 2} \\
        &\leq 4 \bfa ^2 (\v{\phi}_i) ^2 \rbr*{ \int_{[a,b]^d} \abs{ \realapprox{\phi}{\infty} (x) - f ( x ) } \, \mu ( \d x )  } ^{\! \! 2} \leq 4 \bfa ^2 (\v{\phi}_i)^2 \cL_\infty (\phi).
    \end{split}
\end{equation}
Next \nobs that \cref{eq:loss:gradient,lem:grad:est:eq1} prove that for all $i \in \{1,2, \ldots, \width \}$ we have that
\begin{equation} \label{eq:lem:gradient:est2}
\begin{split}
         \abs{ \cG_{\width d + i}( \phi) }^2 &= 4 (\v{\phi}_i) ^2 \rbr*{ \int_{I_i^\phi} (\realapprox{\phi}{\infty} (x) - f ( x )  ) \, \mu ( \d x )  } ^{\! \! 2} \\
         &\leq 4 (\v{\phi}_i) ^2  \rbr*{ \int_{[a,b]^d}   \abs{\realapprox{\phi}{\infty} (x) - f ( x ) } \, \mu ( \d x )  } ^{ \! \! 2} \leq 4 (\v{\phi}_i)^2 \cL_\infty (\phi).
         \end{split}
\end{equation}
Furthermore, \nobs that the fact that for all $x = (x_1, \ldots, x_d) \in [a,b]^d$, $i \in \{1,2, \ldots, \width \}$ it holds that $\abs{ \Rect _\infty \rbr{ \b{\phi}_i + \smallsum_{j = 1}^d \w{\phi}_{i,j} x_j } } ^2 \leq  \rbr{ \abs{ \b{\phi}_i } + \bfa  \smallsum_{j = 1}^d \abs{\w{\phi}_{i,j} } } ^2 \leq \bfa^2 (d+1) \rbr{ \abs{ \b{\phi}_i }^2 + \smallsum_{j = 1}^d \abs{\w{\phi}_{i,j} }^2  }$ and \cref{eq:loss:gradient} assure that for all $i \in \{1,2, \ldots, \width \}$ it holds that
\begin{equation} \label{eq:lem:gradient:est3}
\begin{split}
    \abs{ \cG_{ \width ( d+1 )  + i} ( \phi ) } ^2 &= 4 \rbr*{ \int_{[a,b]^d} \br[\big]{\Rect _\infty \rbr[\big]{ \b{\phi}_i + \smallsum_{j = 1}^d \w{\phi}_{i,j} x_j } } ( \realapprox{\phi}{\infty}(x) - f ( x )  ) \, \mu ( \d x )  } ^{\! \! 2 } \\ 
    &\leq 4 \int_{[a,b]^d} \abs[\big]{ \Rect _\infty \rbr[\big]{ \b{\phi}_i + \smallsum_{j = 1}^d \w{\phi}_{i,j} x_j } } ^2 ( \realapprox{\phi}{\infty}(x) - f ( x )  ) ^2 \, \mu ( \d x ) \\
    &\leq 4 \bfa ^2 (d+1) \br*{ \abs{ \b{\phi}_i }^2 + \smallsum_{j = 1}^d \abs{\w{\phi}_{i,j} }^2 } \cL_\infty (\phi).
    \end{split}
\end{equation}
Moreover, \nobs that \cref{eq:loss:gradient,lem:grad:est:eq1} show that
\begin{equation} \label{eq:lem:gradient:est4}
    \abs{ \cG _ { \fd } ( \phi ) } ^2 = 4 \rbr*{ \int_{[a,b]^d} (\realapprox{\phi}{\infty} (x) - f ( x )   ) \, \mu ( \d x )  } ^{\! \! 2 } \leq 4 \cL_\infty ( \phi ).
\end{equation}
Combining this with \cref{eq:lem:gradient:est1}, \cref{eq:lem:gradient:est2}, and \cref{eq:lem:gradient:est3} ensures that
\begin{equation}
\begin{split}
    &\norm{ \cG ( \phi ) } ^2 \\
    &\leq 4 \br*{ \smallsum_{i = 1}^\width \rbr*{\bfa^2 \br*{ \sum_{j = 1}^d  \abs{\v{\phi}_i} ^2 } + \abs{\v{\phi}_i} ^2 + \bfa^2 (d+1) \br*{ \abs{\b{\phi}_i} ^2 + \smallsum_{j = 1}^d \abs{\w{\phi}_{i,j}} ^2  } } } 
    \cL_\infty ( \phi )  + 4 \cL_\infty ( \phi ) \\
    &\leq 4(\bfa^2(d+1) \norm{ \phi } ^2 + 1) \cL_\infty ( \phi ).
    \end{split}
\end{equation}
\end{cproof}

\cfclear
\begin{cor} \label{cor:g:bounded} 
Assume \cref{setting:snn} and let $K \subseteq \R^{\fd}$ be compact. Then $\sup_{\phi \in K} \norm{ \cG ( \phi ) } < \infty$.
\end{cor}
\begin{proof} [Proof of \cref{cor:g:bounded}]
Observe that \cref{lem:realization:lip} and the assumption that $K$ is compact ensure that $\sup_{\phi \in K} \cL_\infty ( \phi ) < \infty$. This and \cref{lem:gradient:est} complete the proof of \cref{cor:g:bounded}.
\end{proof}

\subsection{Upper estimates associated to Lyapunov functions}
\label{subsection:lyapunov:elementary}

\begin{lemma} \label{prop:lyapunov:norm} 
Let $\fd \in \N$, $\xi \in \R$ and let $\norm{ \cdot } \colon \R^\fd \to \R$ and $V \colon \R^\fd \to \R$ satisfy for all $\phi = (\phi_1, \ldots, \phi_\fd ) \in \R^\fd$ that $\norm{\phi} = [ \sum_{i=1}^\fd \abs*{ \phi_i } ^2 ] ^{1/2}$ and $V ( \phi ) = \norm{\phi}^2 + \abs{ \phi_\fd - 2 \xi } ^2$. Then it holds for all $\phi \in \R^{\fd}$ that 
\begin{equation} \label{prop:lyapunov:norm:eq1}
    \norm{ \phi } ^2 \leq V(\phi) \leq 3 \norm{ \phi } ^2 + 8 \xi ^2.
\end{equation}
\end{lemma}
\begin{cproof}{prop:lyapunov:norm}
\Nobs that the fact that for all $\phi \in \R^{\fd}$ it holds that $\abs{ \phi_\fd - 2 \xi } ^2 \geq 0$ assures that for all $\phi \in \R^{\fd}$ we have that
\begin{equation} \label{prop:lyapunov:norm:eq2}
V(\phi) = \norm{ \phi } ^2 + \abs{ \phi_\fd - 2 \xi } ^2 \geq \norm{ \phi } ^2.
\end{equation}
Furthermore, \nobs that the fact that for all $x , y \in \R$ it holds that $(x - y )^2 \leq 2(x^2 + y^2)$ ensures that for all $\phi \in \R^{\fd}$ it holds that
\begin{equation}
    V (\phi) \leq \norm{ \phi } ^2 + 2 (\phi_\fd) ^2 + 8 \xi ^2 \leq 3 \norm{ \phi } ^2 + 8 \xi ^2.
\end{equation}
Combining this with \cref{prop:lyapunov:norm:eq2} establishes \cref{prop:lyapunov:norm:eq1}. 
\end{cproof}

\begin{prop} \label{prop:v:gradient}
Let $\fd \in \N$, $\xi \in \R$ and let $V \colon \R^\fd \to \R$ satisfy for all $\phi = ( \phi_1, \ldots, \phi_\fd) \in \R^\fd$ that $V ( \phi ) = \br{  \sum_{i=1}^\fd \abs{\phi_i}^2 } + \abs{\phi_\fd - 2 \xi } ^2$. Then
\begin{enumerate} [label = (\roman*)]
    \item \label{prop:v:gradient:item1} it holds for all $\phi  = ( \phi_1, \ldots, \phi_\fd) \in \R^\fd$ that
    \begin{equation}
    (\nabla V ) ( \phi) = 2 \phi + \rbr[\big]{  0, 0, \ldots, 0, 2 \br{ \phi_\fd - 2 \xi } } 
\end{equation}
and
\item \label{prop:v:gradient:item2} it holds for all $\phi = ( \phi_1, \ldots, \phi_\fd)$, $\psi = ( \psi_1, \ldots, \psi_\fd) \in \R^{\fd}$ that
\begin{equation} \label{eq:prop:v:gradient}
    (\nabla V)(\phi) - (\nabla V)(\psi) = 2(\phi - \psi ) + \rbr[\big]{ 0, 0, \ldots, 0, 2 (\phi_\fd - \psi_\d ) }.
\end{equation}
\end{enumerate}
\end{prop}
\begin{cproof} {prop:v:gradient}
\Nobs that the assumption that for all $\phi \in \R^{\fd}$ it holds that $V ( \phi ) = \sum_{i=1}^\fd \abs{\phi_i}^2 + \abs{\phi_\fd - 2 \xi } ^2$ proves \cref{prop:v:gradient:item1}.
Moreover, \nobs that \cref{prop:v:gradient:item1} establishes \cref{prop:v:gradient:item2}.
\end{cproof}

\begin{prop} \label{prop:lyapunov:gradient} 
Assume \cref{setting:snn}
and let $\phi \in \R^{\fd}$. Then
\begin{equation}
    \langle (\nabla V ) ( \phi), \cG ( \phi ) \rangle = 8 \int_{[a,b]^d} ( \realapprox{\phi}{\infty}(x) - f(0) ) ( \realapprox{\phi}{\infty}( x ) - f ( x ) ) \, \mu ( \d x ).
\end{equation}
\end{prop}
\begin{cproof2}{prop:lyapunov:gradient}
\Nobs that \cref{prop:v:gradient} demonstrates that
\begin{align}
    &(\nabla V) ( \phi ) \\
    &= 2 \rbr[\big]{ \w{\phi}_{1,1}, \ldots, \w{\phi}_{1 , d}, \w{\phi}_{2 , 1}, \ldots, \w{\phi}_{2 , d}, \ldots, \w{\phi}_{\width , 1}, \ldots, \w{\phi}_{\width , d },  
     \b{\phi}_1, \ldots, \b{\phi}_{\width}, \v{\phi}_1, \ldots, \v{\phi}_{\width},  2 ( \c{\phi} - f(0) } . \nonumber
\end{align} 
This and \cref{eq:loss:gradient} imply that
\begin{equation}
    \begin{split}
        &\langle (\nabla V) ( \phi) , \cG(\phi) \rangle \\ 
        &= 4 \br*{ \sum_{i = 1}^\width \sum_{j = 1}^d \rbr*{ \w{\phi}_{i,j} \v{\phi}_i \int_{I_i^\phi} x_j (\realapprox{\phi}{\infty} (x) - f(x) ) \, \mu ( \d x ) } } \\
        &+ \br*{ 4 \sum_{i = 1}^\width \rbr*{ \b{\phi}_i \v{\phi}_i \int_{I_i^\phi}  (\realapprox{\phi}{\infty} (x) - f(x) ) \, \mu ( \d x ) } } \\
        &+ 4 \br*{ \sum_{i = 1}^\width \rbr*{ \v{\phi}_i \int_{[a,b]^d} \br[\big]{ \Rect _\infty \rbr[\big]{ \b{\phi}_i + \smallsum_{j = 1}^d \w{\phi}_{i,j} x_j } } ( \realapprox{\phi}{\infty}(x) - f ( x ) ) \, \mu ( \d x ) } } \\
        &+ 8(\c{\phi} - f ( 0 )) \br*{ \int_{[a,b]^d} (\realapprox{\phi}{\infty} (x) - f(x) ) \, \mu ( \d x ) }.
    \end{split}
\end{equation}
Hence, we obtain that
\begin{equation} 
    \begin{split} 
     &\langle (\nabla V) ( \phi) , \cG(\phi) \rangle \\ 
        &= 4 \br*{ \sum_{i = 1}^\width \rbr*{ \v{\phi}_i \int_{I_i^\phi} \rbr[\big]{ \b{\phi}_i + \smallsum_{j = 1}^d \w{\phi}_{i,j} x_j } (\realapprox{\phi}{\infty} (x) - f(x) ) \, \mu ( \d x ) } } \\
         &+ 4 \br*{ \sum_{i = 1}^\width \rbr*{ \v{\phi}_i \int_{[a,b]^d} \br[\big]{ \Rect _\infty \rbr[\big]{ \b{\phi}_i + \smallsum_{j = 1}^d \w{\phi}_{i,j} x_j } } ( \realapprox{\phi}{\infty}(x) - f ( x ) ) \, \mu ( \d x ) } } \\
        &+ 8 \int_{[a,b]^d} (\c{\phi} - f ( 0 )) (\realapprox{\phi}{\infty} (x) - f(x) ) \, \mu ( \d x ) \\
        & = 8 \int_{[a,b]^d} \rbr*{ (\c{\phi} - f(0) ) + \smallsum_{i = 1}^\width \br[\big]{ \v{\phi}_i \br[\big]{ \Rect _\infty \rbr[\big]{ \b{\phi}_i + \smallsum_{j = 1}^d \w{\phi}_{i,j} x_j } } } } (\realapprox{\phi}{\infty} (x) - f(x) ) \, \mu ( \d x ) \\
        &= 8 \int_{[a,b]^d} ( \realapprox{\phi}{\infty}(x) - f(0) ) ( \realapprox{\phi}{\infty}( x ) - f ( x ) ) \, \mu ( \d x ).
    \end{split}
\end{equation}
\end{cproof2}

\begin{cor} \label{cor:lyapunov:const}
Assume \cref{setting:snn}, assume for all $x \in [a,b]^d$ that $f(x) = f(0)$,
and let $\phi \in \R^{\fd}$. Then $\langle (\nabla V ) ( \phi ) , \cG ( \phi ) \rangle = 8 \cL_\infty ( \phi)$.
\end{cor}
\begin{proof} [Proof of \cref{cor:lyapunov:const}]
\Nobs that the fact that for all $x \in [a,b]^d$ it holds that $f(x) = f(0)$ implies that
\begin{equation}
    \cL _ \infty ( \phi ) = \int_{[a,b]^d} ( \realapprox{\phi}{\infty}(x) - f(0) ) ( \realapprox{\phi}{\infty}( x ) - f ( x ) ) \, \mu ( \d x ).
\end{equation}
Combining this with \cref{prop:lyapunov:gradient} completes the proof of \cref{cor:lyapunov:const}.
\end{proof}

\begin{cor} \label{cor:critical:points} 
Assume \cref{setting:snn}, assume for all $x \in [a,b]^d$ that $f(x) = f(0)$, and let $\phi \in \R^{\fd}$. Then it holds that $ \cG(\phi) = 0$ if and only if $\cL_\infty ( \phi ) = 0 $.
\end{cor}

\begin{cproof} {cor:critical:points}
\Nobs that \cref{cor:lyapunov:const} implies that for all $\varphi \in \R^\fd$ with $\cG ( \varphi ) = 0$ it holds that $ \cL_\infty (\varphi) = \frac{1}{8} \langle (\nabla V ) ( \varphi) , \cG(\varphi) \rangle = 0$.
Moreover, \nobs that the fact that for all $\varphi \in \R^\fd$ it holds that $\cL_\infty (\varphi) = \int_{[a,b]^d} (\realapprox{\varphi}{\infty} (x) - f(0) )^2 \, \mu ( \d x ) $ ensures that for all $\varphi \in \R^\fd$ with $\cL_\infty ( \varphi ) = 0$ we have that
\begin{equation}
    \int_{[a,b]^d} (\realapprox{\varphi}{\infty} (x) - f(0) )^2 \, \mu ( \d x ) = 0.
\end{equation}
This shows that for all $\varphi \in \cu{ \psi \in \R^\fd \colon \rbr{ \cL _ \infty ( \psi ) = 0 } }$ and $\mu$-almost all $x \in [a,b]^d$ it holds that $\realapprox{\varphi}{\infty}(x) = f(0) $.
Combining this with \cref{eq:loss:gradient} demonstrates that for all $\varphi \in \cu{ \psi \in \R^\fd \colon \rbr{ \cL _ \infty ( \psi ) = 0 } }$ we have that $\cG(\varphi) = 0$.
\end{cproof}

\subsection{Lyapunov type estimates for GD processes}
\label{subsection:lyapunov:gd}

\cfclear
\begin{lemma} \label{lem:gradient:descent:gen}
Assume \cref{setting:snn}, assume for all $x \in [a,b]^d$ that $f(x) = f(0)$, and let $\gamma \in [0, \infty)$, $\theta \in \R^\fd$. Then
\begin{equation}
    V ( \theta - \gamma \cG ( \theta ) ) - V ( \theta ) = \gamma^2 \norm{\cG ( \theta ) }^2 + \gamma^2 \abs{\cG_\fd ( \theta ) } ^2 - 8 \gamma \cL_\infty ( \theta ) \leq 2 \gamma^2 \norm{\cG ( \theta ) } ^2 - 8 \gamma \cL_\infty ( \theta ).
\end{equation}
\end{lemma}
\begin{cproof} {lem:gradient:descent:gen}
Throughout this proof let $\bfe \in \R^\fd$ satisfy $\bfe = ( 0 , 0 , \ldots, 0 , 1)$ and let $g \colon \R \to \R$ satisfy for all $t \in \R$ that
\begin{equation} \label{proof:lem:descent:eq1}
    g ( t ) = V ( \theta - t \cG ( \theta ) ).
\end{equation}
\Nobs that \cref{proof:lem:descent:eq1} and the fundamental theorem of calculus prove that
\begin{equation}
    \begin{split}
        V ( \theta - \gamma \cG ( \theta ) ) 
        &= g ( \gamma ) = g ( 0 ) + \int_0^\gamma g'(t) \, \d t = g ( 0 ) + \int_0^\gamma \langle (\nabla V) ( \theta - t \cG ( \theta ) ) , ( - \cG ( \theta ) ) \rangle \, \d t \\
        &= V ( \theta ) - \int_0^\gamma \langle ( \nabla V ) ( \theta - t \cG ( \theta ) ) , \cG ( \theta ) \rangle \, \d t.
    \end{split}
\end{equation}
\cref{cor:lyapunov:const} hence demonstrates that
\begin{equation}
    \begin{split}
         V ( \theta - \gamma \cG ( \theta ) ) 
        &= V ( \theta ) - \int_0^\gamma \langle ( \nabla V ) ( \theta ) , \cG ( \theta ) \rangle \, \d t \\
        & \quad + \int_0^\gamma \langle ( \nabla V ) ( \theta ) - ( \nabla V ) ( \theta - t \cG ( \theta ) ) , \cG ( \theta ) \rangle \, \d t \\
        &= V ( \theta ) - 8 \gamma \cL_\infty ( \theta ) + \int_0^\gamma \langle ( \nabla V ) ( \theta ) - ( \nabla V ) ( \theta - t \cG ( \theta ) ) , \cG ( \theta ) \rangle \, \d t.
    \end{split}
\end{equation}
\cref{prop:v:gradient} therefore proves that
\begin{equation}
    \begin{split}
         V ( \theta - \gamma \cG ( \theta ) ) 
        &= V ( \theta ) - 8 \gamma \cL_\infty ( \theta ) + \int_0^\gamma \langle 2 t \cG ( \theta ) + 2 \c{t \cG ( \theta )} \bfe , \cG ( \theta ) \rangle \, \d t \\
        &= V ( \theta ) - 8 \gamma \cL_\infty ( \theta ) + 2 \norm{\cG ( \theta ) } ^2  \br*{ \int_0^\gamma t \, \d t }
        + 2 \br*{ \int_0^\gamma \rbr[\big]{ \c{t \cG ( \theta )} \langle \bfe , \cG ( \theta ) \rangle } \, \d t }.
    \end{split}
\end{equation}
Hence, we obtain that
\begin{equation}
    \begin{split}
          V ( \theta - \gamma \cG ( \theta ) ) 
        &= V ( \theta ) - 8 \gamma \cL_\infty ( \theta ) + \gamma^2 \norm{\cG ( \theta ) } ^2 + 2 \abs{\langle \bfe , \cG ( \theta ) \rangle }^2 \br*{ \int_0^\gamma t \, \d t } \\
        &= V ( \theta ) - 8 \gamma \cL_\infty ( \theta ) + \gamma^2 \norm{\cG ( \theta ) } ^2 + \gamma^2 \abs{\langle \bfe , \cG ( \theta ) \rangle }^2 \\
        &= V ( \theta ) - 8 \gamma \cL_\infty ( \theta ) + \gamma^2 \norm{\cG ( \theta ) } ^2 + \gamma^2 \abs{\cG_\fd ( \theta ) } ^2.
    \end{split}
\end{equation}
\end{cproof}

\begin{cor} \label{cor:gradient:descent:gen}
Assume \cref{setting:snn}, assume for all $x \in [a,b]^d$ that $f(x) = f(0)$, and let $\gamma \in [0, \infty)$, $\theta \in \R^\fd$. Then
\begin{equation}
    V ( \theta - \gamma \cG ( \theta ) ) - V ( \theta ) \leq 8 \rbr*{\gamma^2 \br*{\bfa^2 ( d+1 ) V ( \theta ) + 1 } - \gamma } \cL_\infty ( \theta ).
\end{equation}
\end{cor}
\begin{cproof}{cor:gradient:descent:gen}
\Nobs that \cref{lem:gradient:est} and \cref{prop:lyapunov:norm} demonstrate that
\begin{equation}
    \norm{\cG ( \theta ) } ^2 \leq 4 \br*{\bfa^2 ( d+1 )\norm{\theta }^2 + 1 } \cL_\infty ( \theta ) \leq 4 \br*{\bfa^2 ( d+1 ) V ( \theta ) + 1 } \cL_\infty ( \theta ) .
\end{equation}
\cref{lem:gradient:descent:gen} therefore shows that
\begin{equation}
    \begin{split}
         V ( \theta - \gamma \cG ( \theta ) ) - V ( \theta )
         & \leq  8 \gamma^2 \br*{\bfa^2 ( d+1 ) V ( \theta ) + 1 } \cL_\infty ( \theta ) - 8 \gamma \cL_\infty (\theta ) \\
         &= 8 \rbr*{\gamma^2 \br*{\bfa^2 ( d+1 ) V ( \theta ) + 1 } - \gamma } \cL_\infty ( \theta ).
    \end{split}
\end{equation}
\end{cproof}

\cfclear
\begin{cor} \label{cor:est:vtheta_n}
Assume \cref{setting:snn}, assume for all $x \in [a,b]^d$ that $f(x) = f(0)$, let $( \gamma_n )_{n \in \N_0} \subseteq [ 0, \infty)$, let $(\Theta_n)_{n \in \N_0} \colon \N_0 \to \R^{\fd }$ satisfy for all $n \in \N_0$ that $\Theta_{n+1} = \Theta_n - \gamma_n \cG ( \Theta_n)$, and let $n \in \N_0$. Then 
\begin{equation} \label{cor:est:vtheta_n:eq1}
      V(\Theta_{n+1}) -  V ( \Theta_n) \leq 8 \rbr*{ (\gamma_n) ^2 \br{ \bfa^2 (d+1) V(\Theta_n) + 1}  -  \gamma_n } \cL_\infty ( \Theta _ n) .
\end{equation}
\end{cor}
\begin{cproof} {cor:est:vtheta_n}
\Nobs that \cref{cor:gradient:descent:gen} establishes \cref{cor:est:vtheta_n:eq1}.
\end{cproof}

\cfclear
\begin{lemma} \label{lem:vthetan:decreasing} 
Assume \cref{setting:snn}, let $( \gamma_n )_{n \in \N_0} \subseteq [ 0, \infty)$, let $(\Theta_n)_{n \in \N_0} \colon \N_0 \to \R^{\fd}$ satisfy for all $n \in \N_0$ that $\Theta_{n+1} = \Theta_n - \gamma_n \cG ( \Theta_n)$, assume for all $x \in [a,b]^d$ that $f(x) = f(0)$, and assume $\sup_{n \in \N_0} \gamma_n \leq \br{\bfa^2(d+1)V(\Theta_0) + 1}^{-1} $. Then it holds  for all $n \in \N_0$ that 
\begin{equation} \label{lem:vthetan:decreasing:eq1} 
V (\Theta_{n+1}) - V ( \Theta_n) \leq - 8 \gamma_n \rbr*{ 1- \br{ \sup \nolimits_{m \in \N_0} \gamma_m } \br{\bfa^2 (d+1) V(\Theta_0) + 1} } \cL_\infty (\Theta_n) \leq 0.
\end{equation}
\end{lemma}
\begin{cproof} {lem:vthetan:decreasing}
Throughout this proof let $\supgn  \in \R$ satisfy $\supgn  = \sup_{n \in \N_0} \gamma_n$.
We now prove \cref{lem:vthetan:decreasing:eq1} by induction on $n \in \N_0$. \Nobs that \cref{cor:est:vtheta_n} and the fact that $\gamma_0 \leq \supgn$ imply that
\begin{equation}
\begin{split}
    V(\Theta_1) - V(\Theta_0) 
    &\leq   \rbr*{  - 8 \gamma_0 + 8 ( \gamma_0 )^2 \br{ \bfa^2 (d+1) V(\Theta_0) + 1 }  } \cL_\infty ( \Theta _ 0) \\
    &\leq \rbr*{  - 8 \gamma_0 + 8 \gamma_0 \supgn \br{ \bfa^2 (d+1) V(\Theta_0) + 1 } } \cL_\infty ( \Theta _ 0) \\
    &= - 8 \gamma_0 ( 1-\supgn  \br{\bfa^2 (d+1)V(\Theta_0) + 1} ) \cL_\infty (\Theta_n) \leq 0.
    \end{split}
\end{equation}
This establishes \cref{lem:vthetan:decreasing:eq1} in the base case $n=0$. For the induction step let $n \in \N$ satisfy for all $m \in \{0, 1, \ldots, n-1\}$ that 
\begin{equation} \label{eq:induction:1}
V( \Theta_{m + 1}) - V ( \Theta_{m} ) \leq - 8 \gamma_m ( 1-\supgn  \br{\bfa^2 (d+1) V(\Theta_0) + 1} ) \cL_\infty (\Theta_m) \leq 0.
\end{equation}
\Nobs that \cref{eq:induction:1} shows that $V(\Theta_n) \leq V(\Theta_{n-1}) \leq \cdots  \leq V(\Theta_0)$. The fact that $\gamma_n \leq \supgn $ and \cref{cor:est:vtheta_n} hence demonstrate that
\begin{equation}
    \begin{split}
    V(\Theta_{n+1}) - V(\Theta_n) &\leq \rbr*{  - 8 \gamma_n + 8 ( \gamma_n ) ^2 \br{ \bfa^2 (d+1) V(\Theta_n) + 1 }  }  \cL_\infty ( \Theta _ n) \\
    &\leq  \rbr*{  - 8 \gamma_n + 8 \gamma_n \supgn  \br{\bfa^2 (d+1) V(\Theta_0) + 1}  } \cL_\infty ( \Theta _ n) \\ 
    &= - 8 \gamma_n ( 1-\supgn  \br{\bfa^2 (d+1) V(\Theta_0) + 1} ) \cL_\infty (\Theta_n) \leq 0.
    \end{split}
\end{equation}
Induction therefore establishes \cref{lem:vthetan:decreasing:eq1}.
\end{cproof}

\subsection{Convergence analysis for GD processes in the training of ANNs}
\label{subsection:gd:convergence}
\cfclear
\begin{theorem} \label{theo:gd:loss} 
Assume \cref{setting:snn}, assume for all $x \in [a,b]^d$ that $ f(x) = f(0)$, let $( \gamma_n )_{n \in \N_0} \subseteq [ 0, \infty)$, let $(\Theta_n)_{n \in \N_0} \colon \N_0 \to \R^{ \fd }$ satisfy for all $n \in \N_0$ that $\Theta_{n+1} = \Theta_n - \gamma_n \cG ( \Theta_n)$, and assume $\sup_{n \in \N_0} \gamma_n < \br{\bfa^2(d+1)V(\Theta_0) + 1}^{-1} $ and $\sum_{n=0}^\infty \gamma_n = \infty$. Then 
\begin{enumerate} [label=(\roman*)] 
    \item \label{theo:gd:item1} it holds that $\sup_{n \in \N_0} \norm{ \Theta_n } \leq \br{V(\Theta_0)}^{1/2} < \infty$ and
    \item \label{theo:gd:item2} it holds that $\limsup_{n \to \infty} \cL_\infty (\Theta_n) = 0$.
\end{enumerate}
\end{theorem}
\begin{cproof}{theo:gd:loss}
Throughout this proof let $\eta \in (0, \infty)$ satisfy $\eta = 8( 1- \br{ \sup_{n \in \N_0} \gamma_n } \br{\bfa^2 (d+1) V(\Theta_0) + 1} )$ and let $\varepsilon \in \R$ satisfy $\varepsilon = ( \nicefrac{1}{3} ) [ \min \{ 1 , \limsup_{n \to \infty} \cL_\infty ( \Theta_n ) \} ] $.
\Nobs that \cref{lem:vthetan:decreasing} implies that for all $n \in \N_0$ we have that $V(\Theta_n ) \leq V ( \Theta_{n-1}) \leq \cdots \leq V ( \Theta_0 ) $. 
Combining this and the fact that for all $n \in \N_0$ it holds that $\norm{ \Theta_n } \leq  \br{V ( \Theta_n )}^{1/2}$ establishes \cref{theo:gd:item1}.
Next \nobs that \cref{lem:vthetan:decreasing} implies for all $N \in \N$ that
\begin{equation}
    \eta  \br*{ \sum_{n=0}^{N-1} \gamma_n \cL_\infty (\Theta_n) } \leq \sum_{n=0}^{N-1} \rbr[\big]{ V(\Theta_{n}) - V(\Theta_{n+1}) } = V(\Theta_0) - V( \Theta_N) \leq V(\Theta_0).
\end{equation}
Hence, we have that
\begin{equation} \label{eq:sum:gamman:finite}
    \sum_{n=0}^\infty \br*{ \gamma_n \cL_\infty (\Theta_n) } \leq \frac{V ( \Theta_0 )}{\eta} < \infty.
\end{equation}
This and the assumption that $\sum_{n=0}^\infty \gamma_n = \infty$ ensure that $\liminf_{n \to \infty} \cL_\infty ( \Theta_n ) = 0$.
We intend to complete the proof of \cref{theo:gd:item2} by a contradiction. In the following we thus assume that 
\begin{equation} \label{eq:assumption:limsup>0}
\limsup_{n \to \infty} \cL_\infty ( \Theta_n ) > 0.
\end{equation}
\Nobs that \cref{eq:assumption:limsup>0} implies that 
\begin{equation}
    0 = \liminf_{n \to \infty} \cL_\infty ( \Theta_n ) < \varepsilon < 2 \varepsilon < \limsup_{n \to \infty} \cL_\infty ( \Theta_n ).
\end{equation}
This shows that there exist $(m_k, n_k) \in \N^2$, $k \in \N$, which satisfy for all $k \in \N$ that $m_k < n_k < m_{k+1}$, $ \cL_\infty ( \Theta_{m_k}) > 2 \varepsilon$, and $ \cL_\infty ( \Theta_{n_k}) <  \varepsilon \leq \min_{j \in \N \cap [m_k, n_k ) } \cL_\infty ( \Theta_j )$.
\Nobs that \cref{eq:sum:gamman:finite} and the fact that for all $k \in \N$, $j \in \N \cap [m_k, n_k )$ it holds that $1 \leq \frac{1}{\varepsilon} \cL_\infty ( \Theta_j )$ assure that
\begin{equation} \label{theo:gd:proof1}
    \sum_{k=1}^\infty \sum_{j=m_k}^{n_k - 1} \gamma_j 
    \leq \frac{1}{\varepsilon} \br*{ \sum_{k=1}^\infty \sum_{j=m_k}^{n_k - 1} \rbr*{ \gamma_j \cL_\infty ( \Theta_j ) } }
    \leq \frac{1}{\varepsilon } \br*{ \sum_{j=0}^\infty \rbr*{ \gamma_j \cL_\infty ( \Theta_j ) } }
    < \infty.
\end{equation}
Next \nobs that \cref{cor:g:bounded,theo:gd:item1} ensure that there exists $\fC \in \R$ which satisfies that 
\begin{equation} \label{theo:gd:gradient:bounded:eq}
    \sup\nolimits_{n \in \N_0} \norm{ \cG ( \Theta_n ) } \leq \fC. 
\end{equation}
\Nobs that the triangle inequality, \cref{theo:gd:proof1}, and \cref{theo:gd:gradient:bounded:eq} prove that
\begin{equation} \label{theo:gd:proof2}
    \sum_{k=1}^\infty \norm{ \Theta_{n_k} - \Theta_{m_k} } 
    \leq\sum_{k=1}^\infty \sum_{j=m_k}^{n_k - 1} \norm{ \Theta_{j+1} - \Theta _j }
    = \sum_{k=1}^\infty \sum_{j=m_k}^{n_k - 1} ( \gamma_j \norm{ \cG ( \Theta_j ) } )
    \leq \fC \br*{ \sum_{k=1}^\infty \sum_{j=m_k}^{n_k - 1} \gamma_j }  < \infty.
\end{equation}
Moreover, \nobs that \cref{lem:realization:lip,theo:gd:item1} demonstrate that there exists $\scrL \in \R $ which satisfies for all $m, n \in \N_0$ that $ \abs{ \cL_\infty ( \Theta_m ) - \cL_\infty ( \Theta_n ) } \leq \scrL \norm{ \Theta_m - \Theta_n }$. This and \cref{theo:gd:proof2} show that
\begin{equation}
    \limsup_{k \to \infty} \abs{ \cL_\infty ( \Theta_{n_k} ) - \cL_\infty ( \Theta_{m_k} ) } \leq \limsup_{k \to \infty}  \rbr[\big]{ \scrL  \norm{ \Theta_{n_k} - \Theta_{m_k} } } = 0.
\end{equation}
Combining this and the fact that for all $k \in \N_0$ it holds that $\cL_\infty ( \Theta_{n_k} ) < \varepsilon < 2 \varepsilon < \cL_\infty ( \Theta_{m_k})$ ensures that
\begin{equation}
    0 < \varepsilon \leq \inf_{k \in \N} \abs{ \cL_\infty ( \Theta_{n_k} ) - \cL_\infty ( \Theta_{m_k} ) } \leq \limsup_{k \to \infty} \abs{ \cL_\infty ( \Theta_{n_k} ) - \cL_\infty ( \Theta_{m_k} ) } = 0.
\end{equation}
This contradiction establishes \cref{theo:gd:item2}.
\end{cproof}

\section{Convergence of stochastic gradient descent (SGD) processes}

In this section we establish in \cref{theorem:sgd} in \cref{subsection:sgd:convergence} below that the true risks of SGD processes 
converge in the training of ANNs with ReLU activation to zero if the target function under consideration 
is a constant. 
In this section we thereby transfer the convergence analysis for GD processes 
from \cref{section:gd} above to a convergence analysis for SGD processes.

\cref{theorem:sgd} in \cref{subsection:sgd:convergence} postulates the mathematical setup in \cref{setting:sgd} in \cref{subsection:setting:gd} below. 
In \cref{setting:sgd} we formally introduce, among other things, the constant $ \xi \in \R $ with which the target function coincides, 
the realization functions $\realapprox{\phi}{\infty} \colon \R^d \to \R $, $ \phi \in \R^\fd $, of the considered ANNs 
(see \cref{setting:sgd:eq:realization} in \cref{setting:sgd}), 
the true risk function $ \mathcal{L} \colon \R^\fd \to \R $, 
the sizes $ M_n \in \N $, $ n \in \N_0 $, of the employed 
mini-batches in the SGD optimization method, 
the empirical risk functions $ \mathfrak{L}^n_{ \infty } \colon \R^\fd \times \Omega \to \R $, $ n \in \N _0$, 
a sequence of smooth approximations 
$ \Rect_r \colon \R \to \R$, $ r \in \N $, 
of the ReLU activation function 
(see \cref{setting:sgd:eq:relu} in \cref{setting:sgd}), 
the learning rates $ \gamma_n \in [0, \infty) $, $ n \in \N_0 $, 
used in the SGD optimization method, 
the appropriately generalized gradient functions  
$ \mathfrak{G}^n = ( \fG_1^n, \ldots, \fG_\fd^n) \colon \R^\fd \times \Omega \to \R^\fd $, 
$ n \in \N_0 $, 
associated to the empirical risk functions, 
as well as
the SGD process 
$\Theta = ( \Theta_n )_{ n \in \N_0 } \colon \N_0 \times \Omega \to \R^\fd$.

\Cref{theo:sgd:item2,theo:sgd:item3} in \cref{theorem:sgd} in \cref{subsection:sgd:convergence} below prove that the true risk 
$ \mathcal{L}( \Theta_n ) $ of the 
SGD process $ \Theta \colon \N_0 \times \Omega \to \R^\fd $ 
converges in the almost sure and $ L^1 $-sense to zero 
as the number of stochastic gradient descent steps 
$ n \in \N $ increases to infinity. 
In our proof of \cref{theorem:sgd} we employ 
the elementary local Lipschitz continuity estimate 
for the true risk function in \cref{lem:realization:lip} in \cref{subsection:loss:lipschitz} above, 
the upper estimates for the standard norm of the generalized gradient functions 
$ \mathfrak{G}^n \colon \R^\fd \times \Omega \to \R^\fd $, $ n \in \N_0 $, 
in \cref{lem:empgradient:est} and \cref{lem:empgradient:bounded} in \cref{subsection:empirical:gradient:est} below, 
the elementary representation results for expectations of empirical risks 
of SGD processes in \cref{cor:expected:loss} in \cref{subsection:expectation:risk} below, 
as well as the Lyapunov type estimates for SGD processes 
in \cref{lem:emp:lyapunov}, \cref{lem:sgd:gen}, \cref{lem:lyapunov:sgd}, and \cref{cor:lyapunov:sgd} in \cref{subsection:lyapunov:sgd} below.

Our proof of \cref{lem:empgradient:bounded} uses \cref{lem:realization:lip} and \cref{lem:empgradient:est}. 
Our proof of \cref{lem:empgradient:est}, in turn, uses the elementary representation result 
for the generalized gradient functions 
$ \mathfrak{G}^n \colon \R^\fd \times \Omega \to \R^\fd $, $ n \in \N_0 $, 
in \cref{emp:loss:differentiable} in \cref{subsection:empirical:gradient} below. 
Our proof of \cref{cor:expected:loss} employs 
the elementary representation result for expectations of the empirical risk functions 
in \cref{prop:unbiased} in \cref{subsection:expectation:risk}
and the elementary measurability result in \cref{lem:sgd:measurable} in \cref{subsection:expectation:risk}.

Very roughly speaking, 
\cref{emp:loss:differentiable} in \cref{subsection:empirical:gradient} below transfers \cref{prop:limit:lr} in \cref{subsection:approx:gradient} above to the SGD setting, 
\cref{lem:empgradient:est} in \cref{subsection:empirical:gradient:est} below transfers \cref{lem:gradient:est} in \cref{subsection:gradient:est} above to the SGD setting, 
\cref{lem:empgradient:bounded} in \cref{subsection:empirical:gradient:est} below transfers \cref{cor:g:bounded} in \cref{subsection:gradient:est} above to the SGD setting, 
\cref{lem:emp:lyapunov} in \cref{subsection:lyapunov:sgd} below transfers \cref{cor:lyapunov:const} in \cref{subsection:lyapunov:elementary} above 
to the SGD setting, 
\cref{lem:sgd:gen} in \cref{subsection:lyapunov:sgd} below transfers \cref{lem:gradient:descent:gen} 
in \cref{subsection:lyapunov:gd} above to the SGD setting, 
\cref{lem:lyapunov:sgd} in \cref{subsection:lyapunov:sgd} below transfers 
\cref{cor:est:vtheta_n} in \cref{subsection:lyapunov:gd} above to the SGD setting, 
\cref{cor:lyapunov:sgd} in \cref{subsection:lyapunov:sgd} below transfers \cref{lem:vthetan:decreasing} in \cref{subsection:lyapunov:gd} above to the SGD setting, 
and \cref{theorem:sgd} in \cref{subsection:sgd:convergence} below transfers \cref{theo:gd:loss} in \cref{subsection:gd:convergence} above to the SGD setting. 

\subsection{Description of the SGD optimization method in the training of ANNs}
\label{subsection:sgd:setting}

\begin{setting} \label{setting:sgd}
Let $d, \width, \fd \in \N$, $\xi  , \bfa , a \in \R$, $b \in (a, \infty)$ satisfy $\fd = d\width + 2 \width + 1$ and $\bfa = \max \{ \abs{a}, \abs{b}, 1 \}$,
let $\Rect _r \colon \R \to \R$, $r \in \N \cup \{ \infty\}$, satisfy for all $x \in \R$ that $ \rbr*{ \bigcup_{r \in \N } \{ \Rect _r \}  } \subseteq C^1 ( \R , \R)$, $\Rect _\infty ( x ) = \max \{ x , 0 \}$, and 
\begin{equation} \label{setting:sgd:eq:relu}
\limsup\nolimits_{r \to \infty}  \rbr*{ \abs { \Rect _r ( x ) - \Rect _\infty ( x ) } + \abs { (\Rect _r)' ( x ) - \indicator{(0, \infty)} ( x ) } } = 0,
\end{equation}
let $\fw  = (( \w{\phi} _ {i,j}  )_{(i,j) \in \{1, \ldots, \width \} \times \{1, \ldots, d \} })_{ \phi \in \R^{\fd}} \colon \R^{\fd} \to \R^{ \width \times d}$,
$\fb =  (( \b{\phi} _ 1 , \ldots, \b{\phi} _ \width ))_{ \phi \in \R^{\fd}} \colon \R^{\fd} \to \R^{\width}$,
$\fv = (( \v{\phi} _ 1 , \ldots, \v{\phi} _ \width ))_{ \phi \in \R^{\fd}} \colon \R^{\fd} \to \R^{\width}$, and
$\fc = (\c{\phi})_{\phi \in \R^{\fd }} \colon \R^{\fd} \to \R$
 satisfy for all $\phi  = ( \phi_1 ,  \ldots, \phi_{\fd}) \in \R^{\fd}$, $i \in \{1, 2, \ldots, \width \}$, $j \in \{1, 2, \ldots, d \}$ that $\w{\phi}_{i , j} = \phi_{ (i - 1 ) d + j}$, $\b{\phi}_i = \phi_{\width d + i}$, 
$\v{\phi}_i = \phi_{ \width ( d+1 )  + i}$, and $\c{\phi} = \phi_{\fd}$,
let $\scrN_ r = (\realapprox{\phi}{r})_{\phi \in \R^{\fd } } \colon \R^{\fd } \to C(\R^d , \R)$, $r \in \N \cup \{ \infty \}$, 
satisfy for all $r \in \N \cup \{ \infty \}$, $\phi \in \R^{\fd}$, $x = (x_1, \ldots, x_d) \in \R^d$ that 
\begin{equation} \label{setting:sgd:eq:realization}
    \realapprox{\phi}{r} (x) = \c{\phi} + \smallsum_{i = 1}^\width \v{\phi}_i \Rect _r \rbr[\big]{ \b{\phi}_i + \smallsum_{j = 1}^d \w{\phi}_{i,j} x_j } ,
\end{equation}
let $\norm{ \cdot } \colon \rbr*{  \bigcup_{n \in \N} \R^n  } \to \R$ and $\langle \cdot , \cdot \rangle \colon \rbr*{  \bigcup_{n \in \N} (\R^n \times \R^n )  } \to \R$ satisfy for all $n \in \N$, $x=(x_1, \ldots, x_n)$, $y=(y_1, \ldots, y_n ) \in \R^n $ that $\norm{ x } = [ \sum_{i=1}^n \abs*{ x_i } ^2 ] ^{1/2}$ and $\langle x , y \rangle = \sum_{i=1}^n x_i y_i$,
let $(\Omega , \cF , \P)$ be a probability space,
let $X^{n , m} = (X^{n,m}_1, \ldots, X^{n,m}_d) \colon \Omega \to [a,b]^d$, $n, m \in \N_0$, be i.i.d.\ random variables,
let $\cL  \colon \R^\fd \to \R$, $V \colon \R^{\fd } \to \R$, and $I_i^\phi \subseteq \R^d$, $\phi \in \R^{\fd }$, $i \in \{1, 2, \ldots, \width \}$, satisfy for all 
$\phi \in \R^{\fd}$, $i \in \{1, 2, \ldots, \width \}$ that $\cL  ( \phi) =  \E \br[\big]{( \realapprox{\phi}{\infty} ( X^{0,0} ) - \xi  ) ^2 }$, $V(\phi) = \norm{ \phi } ^2 + \abs{ \c{\phi} -  2 \xi  } ^2$, and 
\begin{equation}
I_i^\phi = \cu[\big]{ x = (x_1, \ldots, x_d) \in [a,b]^d \colon \b{\phi}_i + \smallsum_{j = 1}^d \w{\phi}_{i,j} x_j  > 0 },
\end{equation}
let $(M_n)_{n \in \N_0} \subseteq \N$,
let $\fL^n_r \colon  \R^{\fd} \times \Omega  \to \R$, $n \in \N_0$, $r \in \N \cup \{ \infty \}$, satisfy for all
$n \in \N_0$, $r \in \N \cup \{ \infty \}$, $\phi \in \R^{\fd}$, $\omega \in \Omega$ that 
$\fL^n_r ( \phi , \omega ) = \frac{1}{M_n} \sum_{m=1}^{M_n}( \realapprox{\phi}{r} ( X^{n,m} ( \omega )) - \xi  ) ^2$,
let $\fG^n = (\fG^n_1, \ldots, \fG^n_{\fd}) \colon \R^{\fd} \times \Omega \to \R^{\fd}$, $n \in \N_0$, satisfy for all
$n \in \N_0$, $\phi \in \R^\fd$, $\omega \in  \{  \mathscr{w} \in \Omega \colon  ((\nabla _\phi \fL^n_r ) ( \phi , \mathscr{w} ) )_{r \in \N } \text{ is convergent} \}$ that $\fG^n ( \phi , \omega ) = \lim_{r \to \infty} (\nabla _\phi \fL^n_r ) ( \phi, \omega )$,
let $\Theta = (\Theta_n)_{n \in \N_0} \colon \N_0 \times \Omega \to \R^{ \fd }$
be a stochastic process, 
let $(\gamma_n)_{n \in \N_0} \subseteq [0, \infty)$,
assume that $\Theta_0$ and $( X^{n,m} )_{(n,m) \in ( \N_0 ) ^2}$ are independent,
and assume for all $n \in \N_0$, $\omega \in \Omega$ that
$\Theta_{n+1} ( \omega) = \Theta_n (\omega) - \gamma_n \fG ^{n} ( \Theta_n (\omega), \omega)$.
\end{setting}

\subsection{Properties of the approximating empirical risk functions and their gradients}
\label{subsection:empirical:gradient}

\begin{prop} \label{emp:loss:differentiable}
Assume \cref{setting:sgd} and let $n \in \N_0$, $\phi \in \R^\fd$, $\omega \in \Omega$. Then 
\begin{enumerate} [label=(\roman*)]
\item \label{emp:loss:differentiable:item0} it holds for all $r \in \N$, $i \in \{1, 2, \ldots, \width \}$, $j \in \{1, 2, \ldots, d \}$ that
\begin{equation} \label{eq:empgradient:approx}
    \begin{split}
        &\rbr[\big]{ \tfrac{\partial}{\partial \phi_{(i-1)d+j}} \fL_r^n } ( \phi , \omega ) \\
        &= \frac{2}{M_n} \sum_{m=1}^{M_n} \Bigl[ \v{\phi}_i \br[\big]{ X^{n,m}_j ( \omega ) } \rbr[\big]{ \realapprox{\phi}{r} (X^{n,m} ( \omega)) - \xi }
         \br[\big]{ (\Rect _r)' \rbr[\big]{\b{\phi}_i + \smallsum_{k=1}^d \w{\phi}_{i,k} X^{n,m}_k ( \omega ) } } \Bigr] , \\
        &\rbr[\big]{ \tfrac{\partial}{\partial \phi_{\width d + i}} \fL_r^n } ( \phi , \omega ) \\
        &= \frac{2}{M_n} \sum_{m=1}^{M_n} \Bigl[ \v{\phi}_i \rbr[\big]{ \realapprox{\phi}{r} (X^{n,m} ( \omega)) - \xi } 
        \br[\big]{ (\Rect _r)'  \rbr[\big]{\b{\phi}_i + \smallsum_{k=1}^d \w{\phi}_{i,k} X^{n,m}_k ( \omega ) } } \Bigr] , \\
        &\rbr[\big]{ \tfrac{\partial}{\partial \phi_{\width ( d+1 ) + i}} \fL_r^n } ( \phi , \omega ) \\
        &= \frac{2}{M_n} \sum_{m=1}^{M_n} \Bigl[ \br[\big]{ \Rect _r \rbr[\big]{\b{\phi}_i + \smallsum_{k=1}^d \w{\phi}_{i,k} X^{n,m}_k ( \omega ) } } 
         \rbr[\big]{ \realapprox{\phi}{r} (X^{n,m} ( \omega)) - \xi } \Bigr] , \\
        &\text{and} \qquad \rbr[\big]{ \tfrac{\partial}{\partial \phi_{\fd }} \fL_r^n } ( \phi , \omega ) = \frac{2}{M_n} \sum_{m=1}^{M_n} \br*{ \rbr[\big]{ \realapprox{\phi}{r} (X^{n,m} ( \omega)) - \xi } } ,
    \end{split}
\end{equation}
    \item \label{emp:loss:differentiable:item1} it holds that $\limsup_{r \to \infty} \norm{ (\nabla \fL^n_r )(\phi , \omega) - \fG^n ( \phi , \omega)} = 0$, and
    \item \label{emp:loss:differentiable:item2} it holds for all $i \in \{1, 2, \ldots, \width \}$, $j \in \{1, 2, \ldots, d \}$ that
    \begin{equation} \label{eq:def:sgd}
    \begin{split}
         \fG^n_{ (i - 1 ) d + j } ( \phi , \omega) &=  \frac{2}{M_n} \sum_{m=1}^{M_n} \br*{ \v{\phi}_i  \br[\big]{ X^{n,m}_j ( \omega ) } \rbr[\big]{ \realapprox{\phi}{\infty} (X^{n,m} ( \omega)) - \xi }  \indicator{I_i^\phi} ( X^{n,m} ( \omega) )} , \\
        \fG^n_{ \width d + i} ( \phi , \omega) &=  \frac{2}{M_n} \sum_{m=1}^{M_n} \br*{ \v{\phi}_i   \rbr[\big]{ \realapprox{\phi}{\infty} (X^{n,m} ( \omega)) - \xi }\indicator{I_i^\phi} ( X^{n,m} ( \omega) )  } , \\
        \fG^n_{\width ( d+1 )  + i} ( \phi , \omega ) &=  \frac{2}{M_n} \sum_{m=1}^{M_n}  \br*{ \br[\big]{ \Rect _\infty \rbr[\big]{ \b{\phi}_i + \smallsum_{k=1}^d\w{\phi}_{i, k} X^{n,m}_k ( \omega) } } \rbr[\big]{ \realapprox{\phi}{\infty}( X^{n,m} ( \omega)) - \xi } } , \\
        \text{and} \qquad \fG^n_{ \fd } ( \phi , \omega ) &=  \frac{2}{M_n} \sum_{m=1}^{M_n}  \rbr[\big]{ \realapprox{\phi}{\infty} ( X^{n,m} ( \omega) ) - \xi  } .
    \end{split}
\end{equation}
\end{enumerate}
\end{prop}
\begin{cproof} {emp:loss:differentiable}
\Nobs that the assumption that for all $r \in \N$ it holds that $\Rect _r \in C^1( \R , \R)$ and the chain rule prove \cref{emp:loss:differentiable:item0}. Next \nobs that \cref{emp:loss:differentiable:item0} and the assumption that for all $x \in \R$ we have that $\limsup_{r \to \infty} \rbr{ \abs { \Rect _r ( x ) - \Rect _\infty ( x ) } + \abs { (\Rect _r)' ( x ) - \indicator{(0, \infty)} ( x ) }} = 0$ establish \cref{emp:loss:differentiable:item1,emp:loss:differentiable:item2}.
\end{cproof}

\subsection{Properties of the expectations of  the empirical risk functions}
\label{subsection:expectation:risk}

\begin{prop} \label{prop:unbiased}
Assume \cref{setting:sgd}. Then
it holds for all $n \in \N_0$, $\phi \in \R^{\fd}$ that $\E [ \fL ^n_\infty ( \phi ) ] = \cL  ( \phi )$.
\end{prop}
\begin{cproof}{prop:unbiased}
\Nobs that the assumption that $X^{n,m} \colon \Omega \to [a,b]^d$, $n,m \in \N_0$, are i.i.d.\ random variables ensures that for all $n \in \N_0$, $\phi \in \R^{\fd}$ it holds that
\begin{equation}
   \E [ \fL_\infty ^n ( \phi ) ] = \tfrac{1}{M_n} \smallsum_{m=1}^{M_n}  \E \br[\big]{ ( \realapprox{\phi}{\infty} ( X^{n,m} ) - \xi  ) ^2 } = \E \br[\big]{ ( \realapprox{\phi}{\infty} ( X^{0,0} ) - \xi  ) ^2} = \cL  ( \phi ).
\end{equation}
\end{cproof}

\begin{lemma} \label{lem:sgd:measurable}
Assume \cref{setting:sgd} and let $\bbF_n \subseteq \cF$, $n \in \N_0$, satisfy for all $n \in \N$ that $\bbF_0 = \sigma ( \Theta_0)$ and $\bbF_n = \sigma \rbr[\big]{ \Theta_0 , \rbr{ X^{\fn, \fm}}_{(\fn , \fm) \in (\N \cap [0,n) ) \times \N_0  } }$. Then
\begin{enumerate} [label=(\roman*)]
    \item \label{lem:sgd:measurable:item0} it holds for all $n \in \N_0$ that $\R^{\fd} \times \Omega \ni (\phi , \omega) \mapsto \fG^n ( \phi , \omega) \in \R^\fd$ is $(\cB(\R^\fd ) \otimes \bbF_{n+1} )/\cB(\R^\fd)$-measurable,
    \item \label{lem:sgd:measurable:item1} it holds for all $n \in \N_0$ that $\Theta_n$ is $\bbF_n/\cB(\R^\fd)$-measurable,
    and
    \item \label{lem:sgd:measurable:item2} it holds for all $m, n \in \N_0$ that $\sigma ( X^{n , m} )$ and $\bbF_n$ are independent.
\end{enumerate}
\end{lemma}
\begin{cproof}{lem:sgd:measurable}
\Nobs that \cref{lem:realization:lip} and \cref{eq:empgradient:approx} prove that for all $n \in \N_0$, $r \in \N$, $\omega \in \Omega$ it holds that $\R^\fd \ni \phi \mapsto (\nabla _\phi \fL_r^n ) ( \phi , \omega ) \in \R^\fd$ is continuous.
Furthermore, \nobs that \cref{eq:empgradient:approx} and the fact that for all $n, m \in \N_0$ it holds that $X^{n,m}$ is $\bbF_{n+1}/\cB ( [a,b]^d)$-measurable assure that for all $n \in \N_0$, $r \in \N$, $\phi \in \R^\fd$ it holds that $\Omega \ni \omega \mapsto (\nabla _\phi \fL_r^n ) ( \phi , \omega ) \in \R^\fd$ is $\bbF_{n+1}/\cB(\R^\fd)$-measurable.
This and, e.g., \cite[Lemma 2.4]{BeckBeckerGrohsJaafariJentzen2018arXiv} show that for all $n \in \N_0$, $r \in \N$ it holds that $\R^{\fd} \times \Omega \ni (\phi , \omega) \mapsto ( \nabla_\phi \fL_r^n ) ( \phi , \omega) \in \R^\fd$ is $(\cB(\R^\fd ) \otimes \bbF_{n+1} )/\cB(\R^\fd)$-measurable.
Combining this with \cref{emp:loss:differentiable:item1} in \cref{emp:loss:differentiable} demonstrates that for all $n \in \N_0$ it holds that 
\begin{equation} 
\R^{\fd} \times \Omega \ni (\phi , \omega) \mapsto \fG^n ( \phi , \omega) \in \R^\fd
\end{equation} 
is $(\cB(\R^\fd ) \otimes \bbF_{n+1} )/\cB(\R^\fd)$-measurable.
This establishes \cref{lem:sgd:measurable:item0}.
In the next step we prove \cref{lem:sgd:measurable:item1} by induction on $n \in \N_0$.
\Nobs that the fact that $\bbF_0 = \sigma  ( \Theta_0)$ ensures that $\Theta_0$ is $\bbF_0/\cB(\R^\fd)$-measurable.
For the induction step let $n \in \N_0$ satisfy that $\Theta_n$ is $\bbF_n/\cB ( \R^\fd)$-measurable.
\Nobs that \cref{lem:sgd:measurable:item0} and the fact that $\bbF_n \subseteq \bbF_{n+1}$ ensure that $\fG^n ( \Theta_n)$ is $ \bbF_{n+1}/\cB(\R^\fd)$-measurable.
Combining this, the fact that $\bbF_n \subseteq \bbF_{n+1}$, and the assumption that $\Theta_{n+1} = \Theta_n - \gamma_n \fG ^{n} ( \Theta_n)$ demonstrates that $\Theta_{n+1}$ is $\bbF_{n+1}/\cB(\R^\fd)$-measurable.
Induction thus establishes \cref{lem:sgd:measurable:item1}.
Next \nobs that the assumption that $X^{n,m}$, $n, m \in \N_0$, are independent and the assumption that $\Theta_0$ and $( X^{n,m} )_{(n,m) \in ( \N_0 ) ^2}$ are independent establish \cref{lem:sgd:measurable:item2}.
\end{cproof}

\begin{cor} \label{cor:expected:loss}
Assume \cref{setting:sgd}. Then it holds for all $n \in \N_0$ that $\E [ \fL_\infty ^n ( \Theta_n ) ] = \E [ \cL  ( \Theta_n ) ]$.
\end{cor}
\begin{cproof} {cor:expected:loss}
Throughout this proof let $\bbF_n \subseteq \cF$, $n \in \N_0$, satisfy for all $n \in \N$ that $\bbF_0 = \sigma ( \Theta_0)$ and $\bbF_n = \sigma \rbr{ \Theta_0 , \rbr{ X^{\fn, \fm}}_{(\fn , \fm) \in (\N \cap [0,n) ) \times \N_0  } }$ and let $\bfL^n \colon ([a,b]^d)^{ M_n} \times \R^{ \fd } \to [0, \infty)$, $n \in \N_0$, satisfy for all $n \in \N_0$, $ x_1, x_2, \ldots, x_{M_n} \in [a,b]^{d }$, $\phi \in \R^{ \fd }$ that 
\begin{equation} \label{cor:expectedloss:eq1}
\bfL^n ( x_1, \ldots, x_{M_n} , \phi ) = \tfrac{1}{M_n} \smallsum_{m=1}^{M_n} ( \realapprox{\phi}{\infty} ( x_m ) - \xi  ) ^ 2.
\end{equation}
\Nobs that \cref{cor:expectedloss:eq1} implies that for all $n \in \N_0$, $\phi \in \R^\fd$, $\omega \in \Omega$ it holds that 
\begin{equation} \label{cor:expectedloss:eq2}
\fL ^n _\infty ( \phi , \omega ) = \bfL^n ( X^{n, 1} ( \omega ), \ldots, X^{n, M_n}(\omega) , \phi).
\end{equation}
Hence, we obtain that for all $n \in \N_0$ it holds that
\begin{equation} \label{cor:expectedloss:eq3}
    \fL_\infty^n ( \Theta_n ) = \bfL^n ( X^{n,1}, \ldots, X^{n , M_n}, \Theta_n ).
\end{equation}
Furthermore, \nobs that \cref{cor:expectedloss:eq2} and \cref{prop:unbiased} imply that for all $n \in \N_0$, $\phi \in \R^{\fd}$ we have that $\E [ \bfL^n ( (X^{n, 1}, \ldots, X^{n, M_n}) , \phi ) ] = \cL  ( \phi)$.
This, \cref{lem:sgd:measurable}, \cref{cor:expectedloss:eq3}, and, e.g., \cite[Lemma 2.8]{JentzenKuckuckNeufeldVonWurstemberger2018arxiv} (applied with $(\Omega, \cF, \P) \with (\Omega, \cF, \P)$, $\cG \with \bbF_n$, $(\bbX, \cX) \with (([a , b]^{d}) ^{ M_n} , \cB (([a , b]^{d}) ^{ M_n}) )$, $(\bbY, \cY) \with ( \R^{\fd}, \cB ( \R^\fd) )$, $X \with (\Omega \ni \omega \mapsto ( X^{n, 1} (\omega), \ldots, X^{n, M_n} ( \omega) ) \in ([a , b]^{d}) ^{ M_n} )$, $Y \with (\Omega \ni \omega \mapsto  \Theta_n ( \omega) \in \R^\fd )$ in the notation of \cite[Lemma 2.8]{JentzenKuckuckNeufeldVonWurstemberger2018arxiv}) demonstrate that for all $n \in \N_0$ it holds that $\E [ \fL_\infty ^n ( \Theta_n ) ] = \E [ \bfL^n ( X^{n, 1}, \ldots, X^{n, M_n} , \Theta_n) ] = \E [ \cL  ( \Theta_n ) ]$.
\end{cproof}

\subsection{Upper estimates for generalized gradients of the empirical risk functions}
\label{subsection:empirical:gradient:est}

\begin{lemma} \label{lem:empgradient:est}
Assume \cref{setting:sgd} and let $n \in \N_0$, $\phi \in \R^\fd$, $\omega \in \Omega$. Then $\norm{ \fG ^n ( \phi , \omega ) } ^2 \leq 4( \bfa^2 (d+1) \norm{ \phi } ^2 + 1 ) \fL_\infty ^n ( \phi , \omega)$.
\end{lemma}
\begin{cproof} {lem:empgradient:est}
\Nobs that Jensen's inequality implies that
\begin{equation} \label{lem:empgrad:est:eq1}
    \rbr*{ \tfrac{1}{M_n} \smallsum_{m=1}^{M_n} \abs{ \realapprox{\phi}{\infty} ( X^{n,m} ( \omega) ) - \xi  } } ^{ \! 2 } \leq \tfrac{1}{M_n} \smallsum_{m=1}^{M_n} (\realapprox{\phi}{\infty} \rbr[\big]{ X^{n,m} ( \omega) ) - \xi }^2 = \fL_\infty ^n ( \phi , \omega ).
\end{equation}
This and \cref{eq:def:sgd} ensure that for all $i \in \{1, 2, \ldots, \width \}$, $j \in \{1,2, \ldots, d \}$ we have that
\begin{equation} \label{eq:lem:empgradient:est1}
    \begin{split}
        \abs{ \fG^n_{ (i - 1 ) d + j }( \phi , \omega ) } ^2 &= (\v{\phi}_i) ^2 \rbr*{ \frac{2}{M_n} \sum_{m=1}^{M_n} \br*{  \br[\big]{ X^{n,m}_j ( \omega ) } \rbr[\big]{ \realapprox{\phi}{\infty} (X^{n,m} ( \omega)) - \xi }  \indicator{I_i^\phi} ( X^{n,m} ( \omega) ) } } ^{\! \! 2 } \\
        & \leq (\v{\phi}_i) ^2 \rbr*{ \frac{2}{M_n} \sum_{m=1}^{M_n} \br*{  \abs[\big]{ X^{n,m}_j ( \omega ) } \abs[\big]{ \realapprox{\phi}{\infty} (X^{n,m} ( \omega)) - \xi }  \indicator{I_i^\phi} ( X^{n,m} ( \omega) ) }  } ^{ \! \! 2} \\
        &\leq 4 \bfa ^2 (\v{\phi}_i) ^2 \rbr*{ \frac{1}{M_n} \sum_{m=1}^{M_n}  \abs[\big]{ \realapprox{\phi}{\infty} (X^{n,m} ( \omega)) - \xi } } ^{\! \! 2} \leq 4 \bfa ^2 (\v{\phi}_i)^2 \fL^n_\infty (\phi , \omega ).
    \end{split}
\end{equation}
In addition, \nobs that \cref{eq:def:sgd,lem:empgrad:est:eq1} assure that for all $i \in \{1,2, \ldots, \width \}$ it holds that
\begin{equation} \label{eq:lem:empgradient:est2}
\begin{split}
         \abs{ \fG^n_{\width d + i}( \phi , \omega ) }^2 &=  (\v{\phi}_i) ^2 \rbr*{ \frac{2}{M_n} \sum_{m=1}^{M_n} \br*{ \rbr[\big]{ \realapprox{\phi}{\infty} (X^{n,m} ( \omega)) - \xi }  \indicator{I_i^\phi} ( X^{n,m} ( \omega) ) } } ^{\! \! 2} \\
         &\leq 4 (\v{\phi}_i) ^2  \rbr*{ \frac{1}{M_n} \sum_{m=1}^{M_n}  \abs[\big]{ \realapprox{\phi}{\infty} (X^{n,m} ( \omega)) - \xi } } ^{\! \! 2} \leq 4 (\v{\phi}_i)^2 \fL^n_\infty (\phi , \omega ).
         \end{split}
\end{equation}
Furthermore, \nobs that for all $x = (x_1, \ldots, x_d) \in [a,b]^d$, $i \in \{1,2, \ldots, \width \}$ it holds that $\abs[\big]{ \Rect _\infty \rbr[\big]{ \b{\phi}_i + \smallsum_{j = 1}^d \w{\phi}_{i,j} x_j } } ^2 \leq  \rbr[\big]{ \abs{ \b{\phi}_i } + \bfa  \smallsum_{j = 1}^d \abs{\w{\phi}_{i,j} } } ^2 \leq \bfa^2 (d+1) \rbr[\big]{ \abs{ \b{\phi}_i }^2 + \smallsum_{j = 1}^d \abs{\w{\phi}_{i,j} }^2  }$. Combining this, the fact that for all $m,n \in \N_0$, $\omega \in \Omega$ it holds that $X^{n,m} ( \omega ) \in [a,b]^d$, \cref{eq:def:sgd}, and Jensen's inequality demonstrates that for all $i \in \{1,2, \ldots, \width \}$ it holds that
\begin{equation} \label{eq:lem:empgradient:est3}
\begin{split}
    \abs{ \fG^n_{ \width ( d+1 )  + i} ( \phi , \omega ) } ^2 &= \rbr*{ \frac{2}{M_n} \sum_{m=1}^{M_n}  \br*{ \br[\big]{ \Rect _\infty \rbr[\big]{ \b{\phi}_i + \smallsum_{k=1}^d\w{\phi}_{i, k} X^{n,m}_k ( \omega) } } \rbr[\big]{ \realapprox{\phi}{\infty}( X^{n,m} ( \omega)) - \xi } }  } ^{\! \! 2 } \\ 
    &\leq \frac{4}{M_n} \sum_{m=1}^{M_n}  \abs[\big]{ \Rect _\infty \rbr[\big]{ \b{\phi}_i + \smallsum_{k=1}^d\w{\phi}_{i, k} X^{n,m}_k ( \omega) } }^2 \rbr[\big]{ \realapprox{\phi}{\infty}( X^{n,m} ( \omega)) - \xi } ^2  \\
    &\leq 4 \bfa ^2 (d+1) \br*{ \abs{ \b{\phi}_i }^2 + \smallsum_{j = 1}^d \abs{\w{\phi}_{i,j} }^2 } \fL^n_\infty (\phi , \omega).
    \end{split}
\end{equation}
Moreover, \nobs that \cref{eq:def:sgd,lem:empgrad:est:eq1} show that
\begin{equation} \label{eq:lem:empgradient:est4}
    \abs{ \fG^n _ { \fd } ( \phi , \omega ) } ^2 = 4 \rbr*{ \tfrac{1}{M_n} \smallsum_{m=1}^{M_n}  \rbr[\big]{ \realapprox{\phi}{\infty} ( X^{n,m} ( \omega) ) - \xi  } } ^{ \! 2 } \leq 4 \fL^n_\infty ( \phi , \omega ).
\end{equation}
Combining \cref{eq:lem:empgradient:est1,eq:lem:empgradient:est2,eq:lem:empgradient:est3,eq:lem:empgradient:est4} yields
\begin{equation}
\begin{split}
    &\norm{ \fG^n ( \phi , \omega ) } ^2 \\
    &\leq 4 \br*{ \smallsum_{i = 1}^\width \rbr*{\bfa^2 \br*{ \sum_{j = 1}^d  \abs{\v{\phi}_i} ^2 } + \abs{\v{\phi}_i} ^2 + \bfa^2 (d+1) \br*{ \abs{\b{\phi}_i} ^2 + \smallsum_{j = 1}^d \abs{\w{\phi}_{i,j}} ^2  } } } 
    \fL^n_\infty ( \phi , \omega ) + 4 \fL^n_\infty ( \phi , \omega ) \\
    &\leq 4 \bfa^2 \br*{\smallsum_{i=1}^\width \rbr*{(d+1) \abs{\v{\phi}_i}^2 + (d+1) \br*{\abs{\b{\phi}_i}^2 + \smallsum_{j = 1}^d \abs{\w{\phi}_{i,j}} ^2} } }  \fL^n_\infty ( \phi , \omega ) + 4 \fL^n_\infty ( \phi , \omega )\\
    & = 4(\bfa^2(d+1) \norm{ \phi } ^2 + 1) \fL^n_\infty ( \phi , \omega ).
    \end{split}
\end{equation}
\end{cproof}

\begin{lemma} \label{lem:empgradient:bounded}
Assume \cref{setting:sgd} and let $K \subseteq \R^\fd$ be compact. Then
\begin{equation} 
\sup\nolimits_{n \in \N_0} \sup\nolimits_{\phi \in K } \sup\nolimits_{\omega \in \Omega} \norm{ \fG ^n ( \phi , \omega ) } < \infty.
\end{equation}
\end{lemma}
\begin{proof} [Proof of \cref{lem:empgradient:bounded}]
\Nobs that \cref{lem:realization:lip} proves that there exists $\fC \in \R$ which satisfies for all $\phi \in K$ that $ \sup_{x \in [a,b]^d} \abs{ \realapprox{\phi}{\infty} ( x) } \leq \fC$. The fact that for all $n , m\in \N_0$, $\omega \in \Omega$ it holds that $X^{n , m } (\omega) \in [a,b]^d$ hence establishes that for all $n \in \N_0$, $\phi \in K$, $\omega \in \Omega$ we have that
\begin{equation}
\begin{split}
    \fL_\infty ^n ( \phi , \omega ) &= \tfrac{1}{M_n}\smallsum_{m=1}^{M_n} ( \realapprox{\phi}{\infty} ( X^{n , m} ( \omega )) - \xi  ) ^2 \\
    &\leq \tfrac{2}{M_n}\smallsum_{m=1}^{M_n} \br[\big]{ \abs{ \realapprox{\phi}{\infty} ( (X^{n , m} ( \omega )) }^2 +  \xi  ^2 } \leq 2 \fC ^2 + 2 \xi  ^2.
\end{split}
\end{equation}
Combining this and \cref{lem:empgradient:est} completes the proof of \cref{lem:empgradient:bounded}.
\end{proof}

\subsection{Lyapunov type estimates for SGD processes}
\label{subsection:lyapunov:sgd}

\begin{lemma} \label{lem:emp:lyapunov}
Assume \cref{setting:sgd} and let $n \in \N_0$, $\phi \in \R^\fd$, $\omega \in \Omega$. Then $\langle \nabla V ( \phi ) , \fG ^n ( \phi , \omega ) \rangle = 8 \fL_\infty ^n ( \phi , \omega )$.
\end{lemma}
\begin{cproof}{lem:emp:lyapunov}
\Nobs that the fact that $V(\phi ) = \norm{\phi}^2 + \abs{\c{\phi} - 2 \xi }^2$ ensures that
\begin{align}
    &(\nabla V) ( \phi ) \\
    &= 2 \rbr[\big]{ \w{\phi}_{1,1}, \ldots, \w{\phi}_{1 , d}, \w{\phi}_{2 , 1}, \ldots, \w{\phi}_{2 , d}, \ldots, \w{\phi}_{\width , 1}, \ldots, \w{\phi}_{\width , d },  
     \b{\phi}_1, \ldots, \b{\phi}_{\width}, \v{\phi}_1, \ldots, \v{\phi}_{\width},  2 ( \c{\phi} - \xi ) }  . \nonumber
\end{align} 
This and \cref{eq:def:sgd} imply that
\begin{equation}
    \begin{split}
       & \langle (\nabla V) ( \phi) , \fG^n(\phi , \omega) \rangle \\
        &= \frac{4}{M_n} \br*{ \sum_{i = 1}^\width \sum_{j = 1}^d  \w{\phi}_{i,j} \v{\phi}_i \rbr*{ \sum_{m=1}^{M_n} \br*{ \br[\big]{ X^{n,m}_j ( \omega ) } \rbr[\big]{ \realapprox{\phi}{\infty} (X^{n,m} ( \omega)) - \xi }  \indicator{I_i^\phi} ( X^{n,m} ( \omega) ) } } }\\
        &+ \frac{4}{M_n} \br*{ \sum_{i = 1}^\width \b{\phi}_i \v{\phi}_i \rbr*{ \sum_{m=1}^{M_n} \br*{ \rbr[\big]{ \realapprox{\phi}{\infty} (X^{n,m} ( \omega)) - \xi }  \indicator{I_i^\phi} ( X^{n,m} ( \omega) ) } } } \\
        &+ \frac{4}{M_n} \br*{ \sum_{i = 1}^\width \v{\phi}_i \rbr*{ \sum_{m=1}^{M_n}  \br*{  \br[\big]{ \Rect _\infty \rbr[\big]{ \b{\phi}_i + \smallsum_{k=1}^d\w{\phi}_{i, k} X^{n,m}_k ( \omega) } } \rbr[\big]{ \realapprox{\phi}{\infty}( X^{n,m} ( \omega)) - \xi } } } } \\
        &+ \frac{8 (\c{\phi} - \xi )}{M_n} \br*{ \sum_{m=1}^{M_n} \rbr[\big]{ \realapprox{\phi}{\infty}( X^{n,m} ( \omega)) - \xi } } .
    \end{split}
\end{equation}
Hence, we obtain that
\begin{equation} 
    \begin{split} 
     & \langle (\nabla V) ( \phi) , \fG^n(\phi , \omega) \rangle \\ 
        &= \frac{4}{M_n} \br*{ \sum_{i = 1}^\width \v{\phi}_i \rbr*{ \sum_{m=1}^{M_n} \br*{  \rbr[\big]{ \b{\phi}_i + \smallsum_{j = 1}^d \w{\phi}_{i,j}  X^{n,m}_j ( \omega ) } \rbr[\big]{ \realapprox{\phi}{\infty} (X^{n,m} ( \omega)) - \xi }  \indicator{I_i^\phi} ( X^{n,m} ( \omega) ) } } } \\
         &+ \frac{4}{M_n} \br*{ \sum_{i = 1}^\width \v{\phi}_i \rbr*{ \sum_{m=1}^{M_n}  \br*{  \br[\big]{ \Rect _\infty \rbr[\big]{ \b{\phi}_i + \smallsum_{k=1}^d\w{\phi}_{i, k} X^{n,m}_k ( \omega) } } \rbr[\big]{ \realapprox{\phi}{\infty}( X^{n,m} ( \omega)) - \xi } } } } \\
        &+ \frac{8 (\c{\phi} - \xi )}{M_n} \br*{ \sum_{m=1}^{M_n} \rbr[\big]{ \realapprox{\phi}{\infty}( X^{n,m} ( \omega)) - \xi } } \\
        & = \frac{8}{M_n} \sum_{m=1}^{M_n} \br*{ \rbr*{ (\c{\phi} - \xi) + \smallsum_{i = 1}^\width \br*{ \v{\phi}_i \br[\big]{ \Rect _\infty \rbr[\big]{ \b{\phi}_i + \smallsum_{j = 1}^d \w{\phi}_{i,j} X^{n,m}_j ( \omega ) } } }  } \rbr[\big]{ \realapprox{\phi}{\infty}( X^{n,m} ( \omega)) - \xi } }\\
        &= \frac{8}{M_n} \sum_{m=1}^{M_n} \rbr[\big]{ \realapprox{\phi}{\infty}( X^{n,m} ( \omega)) - \xi }^2 = 8 \fL^n_\infty ( \phi , \omega ).
    \end{split}
\end{equation}
\end{cproof}

\cfclear
\begin{lemma} \label{lem:sgd:gen}
Assume \cref{setting:sgd} and let $n \in \N_0$, $\theta \in \R^\fd$, $\omega \in \Omega$. Then
\begin{equation}
\begin{split}
    V ( \theta - \gamma_n \fG^n ( \theta , \omega ) ) - V ( \theta ) 
    &= (\gamma_n)^2 \norm{\fG^n ( \theta , \omega ) }^2 + (\gamma_n) ^2 \abs{\fG^n_\fd ( \theta , \omega ) } ^2 - 8 \gamma_n \fL^n_\infty ( \theta , \omega ) \\
    &\leq 2 (\gamma_n)^2 \norm{\fG^n ( \theta , \omega ) } ^2 - 8 \gamma_n \fL^n_\infty ( \theta , \omega ).
\end{split}
\end{equation}
\end{lemma}
\begin{cproof} {lem:sgd:gen}
Throughout this proof let $\bfe \in \R^\fd$ satisfy $\bfe = ( 0 , 0 , \ldots, 0 , 1)$ and let $g \colon \R \to \R$ satisfy for all $t \in \R$ that
\begin{equation} \label{proof:lem:sgd:eq1}
    g ( t ) = V ( \theta - t \fG^n ( \theta , \omega ) ).
\end{equation}
\Nobs that \cref{proof:lem:sgd:eq1} and the fundamental theorem of calculus prove that
\begin{equation}
    \begin{split}
        V ( \theta - \gamma_n \fG^n ( \theta , \omega ) ) 
        &= g ( \gamma_n ) = g ( 0 ) + \int_0^{\gamma_n} g'(t) \, \d t \\
        &= g ( 0 ) + \int_0^{\gamma _n} \langle (\nabla V) ( \theta - t \fG^n ( \theta , \omega ) ) , ( - \fG ( \theta , \omega) ) \rangle \, \d t \\
        &= V ( \theta ) - \int_0^{\gamma_n} \langle ( \nabla V ) ( \theta - t \fG^n ( \theta , \omega ) ) , \fG ( \theta , \omega ) \rangle \, \d t.
    \end{split}
\end{equation}
\cref{lem:emp:lyapunov} hence demonstrates that
\begin{equation}
    \begin{split}
        & V ( \theta - \gamma_n \fG^n ( \theta , \omega ) ) \\
        &= V ( \theta ) - \int_0^{\gamma_n} \langle ( \nabla V ) ( \theta ) , \fG^n ( \theta , \omega ) \rangle \, \d t \\
        & \quad + \int_0^{\gamma_n} \langle ( \nabla V ) ( \theta ) - ( \nabla V ) ( \theta - t \fG^n ( \theta , \omega ) ) , \fG^n ( \theta , \omega ) \rangle \, \d t \\
        &= V ( \theta ) - 8 \gamma_n \fL^n_\infty ( \theta ,\omega ) + \int_0^{\gamma_n} \langle ( \nabla V ) ( \theta ) - ( \nabla V ) ( \theta - t \fG^n ( \theta , \omega ) ) , \fG^n ( \theta , \omega ) \rangle \, \d t.
    \end{split}
\end{equation}
\cref{prop:v:gradient} therefore proves that
\begin{equation}
    \begin{split}
         V ( \theta - \gamma_n \fG^n ( \theta , \omega ) ) 
        &= V ( \theta ) - 8 \gamma_n \fL^n_\infty ( \theta , \omega ) + \int_0^{\gamma_n} \langle 2 t \fG^n ( \theta , \omega ) + 2 \c{t \fG^n ( \theta , \omega )} \bfe , \fG^n ( \theta , \omega ) \rangle \, \d t \\
        &= V ( \theta ) - 8 \gamma_n \fL^n_\infty ( \theta , \omega) + 2 \norm{\fG^n ( \theta , \omega ) } ^2  \br*{ \int_0^{\gamma_n} t \, \d t } \\
        & \quad + 2 \br*{ \int_0^{\gamma_n} \rbr[\big]{ \c{t \fG^n ( \theta , \omega )} \langle \bfe , \fG^n ( \theta , \omega ) \rangle } \, \d t }.
    \end{split}
\end{equation}
Hence, we obtain that
\begin{equation}
    \begin{split}
          &V ( \theta - \gamma_n \fG^n ( \theta , \omega ) ) \\
        &= V ( \theta ) - 8 \gamma_n \fL^n_\infty ( \theta , \omega ) + (\gamma_n) ^2 \norm{\fG^n ( \theta , \omega ) } ^2 + 2 \abs{\langle \bfe , \fG^n ( \theta , \omega ) \rangle }^2 \br*{ \int_0^{\gamma_n} t \, \d t } \\
        &= V ( \theta ) - 8 \gamma_n \fL^n_\infty ( \theta , \omega ) + (\gamma_n)^2 \norm{\fG^n ( \theta , \omega ) } ^2 + (\gamma_n) ^2 \abs{\langle \bfe , \fG^n ( \theta , \omega ) \rangle }^2 \\
        &= V ( \theta ) - 8 \gamma_n \fL^n_\infty ( \theta , \omega ) + (\gamma_n) ^2 \norm{\fG^n ( \theta , \omega ) } ^2 + (\gamma_n) ^2 \abs{\fG^n_\fd ( \theta , \omega ) } ^2.
    \end{split}
\end{equation}
\end{cproof}

\begin{lemma} \label{lem:lyapunov:sgd}
Assume \cref{setting:sgd}. Then it holds for all $n \in \N_0$ that
\begin{equation} \cfadd{def:lyapunov}
      V(\Theta_{n+1}) -  V ( \Theta_n) \leq 8 \rbr*{ (\gamma_n) ^2 \br{ \bfa^2 (d+1) V(\Theta_n) + 1 } - \gamma_n }\fL^n_\infty ( \Theta _ n) .
\end{equation}
\end{lemma}
\begin{cproof}{lem:lyapunov:sgd}
\Nobs that \cref{prop:lyapunov:norm} and \cref{lem:empgradient:est} prove that for all $n \in \N_0$ it holds that
\begin{equation}
    \norm{\fG^n ( \Theta_n ) } ^2 \leq 4 \br*{ \bfa^2 (d+1) \norm{ \Theta_n } ^2 + 1 } \fL_\infty ^n ( \Theta_n ) \leq 4 \br*{ \bfa^2 (d+1) V ( \Theta_n ) + 1 } \fL_\infty ^n ( \Theta_n ).
\end{equation}
\cref{lem:sgd:gen} hence demonstrates that for all $n \in \N_0$ it holds that
\begin{equation}
\begin{split}
    V ( \Theta_{n+1} ) - V ( \Theta_n ) 
    &\leq 2 ( \gamma_n)^2 \norm{\fG^n ( \Theta_n ) } ^2 - 8 \gamma_n \fL_\infty^n ( \Theta_n ) \\
    & \leq 8 (\gamma_n)^2 \br*{ \bfa^2 (d+1) V ( \Theta_n ) + 1 } \fL_\infty ^n ( \Theta_n ) - 8 \gamma_n \fL_\infty^n ( \Theta_n ) \\
    & = 8 \rbr*{ (\gamma_n) ^2 \br{ \bfa^2 (d+1) V(\Theta_n) + 1 } - \gamma_n }\fL^n_\infty ( \Theta _ n) .
    \end{split}
\end{equation}
\end{cproof}

\begin{cor} \label{cor:lyapunov:sgd}
Assume \cref{setting:sgd} and assume $\P \rbr{  \sup_{n \in \N_0} \gamma_n  \leq \br{\bfa^2(d+1)V(\Theta_0) + 1}^{-1} } = 1 $. Then it holds for all $n \in \N_0$ that 
\begin{equation} \label{cor:lyapunov:sgd:eq1}
  \P \rbr[\Big]{  V(\Theta_{n+1} ) - V(\Theta_n) \leq - 8 \gamma_n  \rbr*{ 1 - \br{ \sup\nolimits_{m \in \N_0} \gamma_m } \br{\bfa^2(d+1)V(\Theta_0) + 1} } \fL_\infty ^n ( \Theta_n ) \leq 0 } = 1.
    \end{equation}
\end{cor}
\begin{cproof} {cor:lyapunov:sgd}
Throughout this proof let $\supgn  \in \R$ satisfy $\supgn  = \sup_{n \in \N_0} \gamma_n$. We now prove \cref{cor:lyapunov:sgd:eq1} by induction on $n \in \N_0$. \Nobs that \cref{lem:lyapunov:sgd} and the fact that $\gamma_0 \leq \supgn$ imply that it holds $\P$-a.s.\ that
\begin{equation}
\begin{split}
    V(\Theta_1) - V(\Theta_0) 
    &\leq 8  \rbr*{  ( \gamma_0 )^2 \br{ \bfa^2 (d+1) V(\Theta_0) + 1 } - \gamma_0 } \fL^0_\infty ( \Theta _ 0) \\
    &\leq  8  \rbr*{  \gamma_0 \supgn \br{ \bfa^2 (d+1) V(\Theta_0) + 1 } - \gamma_0 } \fL^0_\infty ( \Theta _ 0) \\
    & = - 8 \gamma_0 \rbr*{  1 - \supgn  \br{\bfa^2(d+1)V(\Theta_0) + 1} } \fL^0_\infty( \Theta _ 0).
    \end{split}
\end{equation}
This establishes \cref{cor:lyapunov:sgd:eq1} in the base case $n=0$. For the induction step let $n \in \N$ satisfy that for all $m \in \{0, 1, \ldots, n-1\}$ it holds $\P$-a.s. that
\begin{equation} \label{eq:induction:sgd1}
V( \Theta_{m + 1}) - V ( \Theta_{m} ) \leq  - 8 \gamma_m  \rbr*{ 1 - \supgn  \br{ \bfa^2 (d+1) V(\Theta_0) + 1} } \fL_\infty ^m ( \Theta_m ) \leq 0.
\end{equation}
\Nobs that \cref{eq:induction:sgd1} shows that it holds $\P$-a.s.\ that $V(\Theta_n) \leq V(\Theta_{n-1}) \leq \cdots  \leq V(\Theta_0)$. The fact that $\gamma_n \leq \supgn $ and \cref{lem:lyapunov:sgd} hence demonstrate that it holds $\P$-a.s.\ that 
\begin{equation}
    \begin{split}
    V(\Theta_{n+1}) - V(\Theta_n) &\leq 8 \rbr*{ ( \gamma_n )^2 \br{ \bfa^2 (d+1)  V(\Theta_n) + 1 }  - \gamma_n }  \fL_\infty^n ( \Theta _ n) \\
    &\leq 8 \rbr*{ \gamma_n \supgn \br{ \bfa^2 (d+1)  V(\Theta_n) + 1 }  - \gamma_n }  \fL_\infty^n ( \Theta _ n) \\
    &=  - 8 \gamma_n  \rbr*{ 1 - \supgn \br{\bfa^2(d+1)V(\Theta_0) + 1} } \fL_\infty ^n ( \Theta_n ) \leq 0.
    \end{split}
\end{equation}
Induction therefore establishes \cref{cor:lyapunov:sgd:eq1}.
\end{cproof}

\subsection{Convergence analysis for SGD processes in the training of ANNs} \label{subsection:sgd:convergence}

\begin{theorem} \label{theorem:sgd}
Assume \cref{setting:sgd}, let $\delta \in (0, 1)$, assume $\sum_{n=0}^\infty \gamma_n = \infty$, and assume for all $n \in \N_0$ that $\P  \rbr*{ \gamma_n  \br{\bfa^2 (d+1) V ( \Theta _ 0 ) \allowbreak + 1} \leq \delta } = 1$. Then
\begin{enumerate} [label=(\roman*)]
\item \label{theo:sgd:item1} there exists $\fC \in \R$ such that $\P  \rbr*{  \sup_{n \in \N_0}  \norm{ \Theta_n } \leq \fC  } = 1$,
    \item \label{theo:sgd:item2} it holds that $\P  \rbr*{  \limsup_{n \to \infty} \cL  ( \Theta_n ) = 0  } = 1$, and
    \item \label{theo:sgd:item3} it holds that $\limsup_{n \to \infty} \E [ \cL  ( \Theta_n ) ] = 0$.
\end{enumerate}
\end{theorem}
\begin{cproof} {theorem:sgd}
Throughout this proof let $\supgn  \in [0 , \infty]$ satisfy $\supgn  = \sup_{n \in \N_0} \gamma_n$.
\Nobs that the assumption that $\delta < 1$, the assumption that $\sum_{n=0}^\infty \gamma_n = \infty$, and the fact that $\P  \rbr*{ \supgn  \br{\bfa^2 (d+1) V ( \Theta _ 0 ) + 1} \leq \delta } = 1$ demonstrates that $\supgn \in (0, \infty)$. Combining this with the fact that $\P  \rbr*{ \supgn  \br{\bfa^2 (d+1) V ( \Theta _ 0 ) + 1} \leq \delta } = 1$ shows that there exists $\fC \in [1 , \infty)$ which satisfies that 
\begin{equation} \label{eq:v:theta0:bounded}
\P ( V(\Theta_0 ) \leq \fC ) = 1.
\end{equation}
\Nobs that \cref{eq:v:theta0:bounded} and \cref{cor:lyapunov:sgd} ensure that $\P  \rbr{ \sup_{ n \in \N_0 } V(\Theta_n) \leq \fC } = 1$.
Combining this and the fact that for all $\phi \in \R^\fd$ it holds that $\norm{ \phi } \leq \br{ V ( \phi ) }^{1/2}$ demonstrates that 
\begin{equation} \label{eq:norm:thetan:bounded}
    \P \rbr*{ \sup\nolimits_{n \in \N_0} \norm{\Theta_n} \leq \fC } = 1.
\end{equation}
This establishes \cref{theo:sgd:item1}.
Next \nobs that \cref{cor:lyapunov:sgd} and the fact that $\P ( \supgn  \br{\bfa^2 (d+1) V ( \Theta _ 0 ) + 1} \leq \delta ) = 1$ prove that for all $n \in \N_0$ it holds $\P$-a.s.\ that
\begin{equation}
    - \rbr[\big]{V(\Theta_{n} ) - V(\Theta_{n+1}) } \leq - 8 \gamma_n  \rbr*{ 1 - \supgn \br{\bfa^2(d+1)V(\Theta_0) + 1} } \fL_\infty ^n ( \Theta_n ).
\end{equation}
This assures that for all $n \in \N_0$ it holds $\P$-a.s.\ that
\begin{equation}
    \gamma_n \fL_\infty ^n ( \Theta_n ) \leq \frac{ V ( \Theta_n ) - V ( \Theta_{n+1})}{ 8 ( 1 - \supgn \br*{ \bfa^2 ( d+1) V ( \Theta_0 ) + 1 } ) }.
\end{equation}
The fact that $\P  \rbr*{ \supgn  \br{\bfa^2 (d+1) V ( \Theta _ 0 ) + 1} \leq \delta } = 1$ and \cref{eq:v:theta0:bounded} hence show that for all $N \in \N$ it holds $\P$-a.s.\ that
\begin{equation}
\begin{split}
    \sum_{n = 0}^{N - 1} \gamma_n \fL_\infty ^n ( \Theta _ n ) &\leq \frac{\smallsum_{n=0}^{N-1} ( V ( \Theta_n ) - V ( \Theta_{n+1} ) )}{8 ( 1 - \supgn \br*{ \bfa^2 ( d+1) V ( \Theta_0 ) + 1 } )} =
    \frac{ V ( \Theta_{0}) - V ( \Theta_N ) }{8 (1 - \supgn  \br{\bfa^2 (d+1) V ( \Theta _ 0 ) + 1} )} \\
    &\leq \frac{V ( \Theta_0) }{8 (1-\delta) } \leq \frac{\fC}{8(1-\delta)} < \infty.
    \end{split}
\end{equation}
This implies that
\begin{equation} \label{proof:theo:sgd:eq1}
    \sum_{n=0}^\infty \gamma_n \E [ \fL_\infty ^n ( \Theta_n )] = \lim_{N \to \infty} \br*{ \sum_{n=0}^{N-1} \gamma_n \E [ \fL_\infty ^n ( \Theta_n ) ] }   \leq \frac{\fC}{8(1-\delta)} < \infty.
\end{equation}
Furthermore, \nobs that \cref{cor:expected:loss} shows for all $n \in \N_0$ that $\E [ \fL _\infty ^n ( \Theta_n )] = \E [ \cL  ( \Theta_n )]$. 
Combining this with \cref{proof:theo:sgd:eq1} proves that
\begin{equation}
    \sum_{n=0}^\infty \E [\gamma_n  \cL  ( \Theta_n ) ] < \infty.
\end{equation}
The monotone convergence theorem and the fact that for all $n \in \N_0$ it holds that $\cL  ( \Theta_n ) \geq 0$ hence demonstrate that
\begin{equation}
    \E  \br*{ \sum_{n=0}^\infty \gamma_n \cL  ( \Theta_n )  } =    \sum_{n=0}^\infty \E [\gamma_n \cL  ( \Theta_n ) ] < \infty.
\end{equation}
Hence, we obtain that $\P \rbr[\big]{ \sum_{n=0}^\infty \gamma_n \cL  ( \Theta_n ) < \infty } = 1$.
Next let $A \subseteq \Omega$ satisfy
\begin{equation}
    A =  \cu*{ \omega \in \Omega \colon \br*{  \rbr*{  \smallsum_{n=0}^\infty \gamma_n \cL  ( \Theta_n (\omega) ) < \infty  } \wedge  \rbr*{  \sup\nolimits_{n \in \N_0} \norm{ \Theta_n ( \omega) } \leq \fC } } }.
\end{equation}
\Nobs that \cref{eq:norm:thetan:bounded} and the fact that $ \P ( \smallsum_{n=0}^\infty \gamma_n \cL  ( \Theta_n) < \infty ) = 1$ prove that $A \in \cF$ and $\P ( A ) = 1$.
In the following let $\omega \in A$ be arbitrary. \Nobs that the assumption that $\sum_{n=0}^\infty \gamma_n = \infty$ and the fact that $ \smallsum_{n=0}^\infty \gamma_n \cL  ( \Theta_n (\omega) ) < \infty$ demonstrate that $\liminf_{n \to \infty} \cL  ( \Theta_n ( \omega)) = 0$. We intend to prove by a contradiction that $ \limsup_{n \to \infty} \cL  ( \Theta_n ( \omega)) = 0$. In the following we thus assume that $\limsup_{n \to \infty} \cL  ( \Theta_n ( \omega)) > 0$. This implies that there exists $\varepsilon \in (0 , \infty)$ which satisfies that 
\begin{equation}
    0 = \liminf_{n \to \infty} \cL  ( \Theta_n ( \omega)) < \varepsilon < 2 \varepsilon < \limsup_{n \to \infty} \cL  ( \Theta_n ( \omega)).
\end{equation}
Hence, we obtain that there exist $(m_k, n_k) \in \N^2$, $k \in \N$, which satisfy for all $k \in \N$ that $m_k < n_k < m_{k+1}$, $ \cL ( \Theta_{m_k} ( \omega ) ) > 2 \varepsilon$, and $ \cL ( \Theta_{n_k} ( \omega ) ) <  \varepsilon \leq \min_{j \in \N \cap [m_k, n_k ) } \cL ( \Theta_j ( \omega ) )$.
\Nobs that the fact that $\sum_{n=0}^\infty \gamma_n \cL  ( \Theta_n (\omega) ) < \infty$ and the fact that for all $k \in \N$, $j \in \N \cap [m_k, n_k )$ it holds that $1 \leq \varepsilon^{-1} \cL ( \Theta_j ( \omega ) )$ assure that
\begin{equation} \label{theo:sgd:proof1}
    \sum_{k=1}^\infty \sum_{j=m_k}^{n_k - 1} \gamma_j 
    \leq \frac{1}{\varepsilon} \br*{ \sum_{k=1}^\infty \sum_{j=m_k}^{n_k - 1} \rbr*{ \gamma_j \cL  ( \Theta_j (\omega) ) } }
    \leq \frac{1}{\varepsilon } \br*{ \sum_{j=0}^\infty \rbr*{ \gamma_j \cL  ( \Theta_j ( \omega ) ) } }
    < \infty.
\end{equation}
Next \nobs that the fact that $\sup\nolimits_{n \in \N_0} \norm{ \Theta_n ( \omega) } \leq \fC$ and \cref{lem:empgradient:bounded} ensure that there exists $\fD \in \R $ which satisfies for all $n \in \N_0$ that $\norm{ \fG^n ( \Theta_n (\omega) , \omega ) } \leq \fD$. Combining this and \cref{theo:sgd:proof1} proves that
\begin{equation}\label{theo:sgd:proof2}
\begin{split}
    \sum_{k=1}^\infty \norm{ \Theta_{n_k}(\omega) - \Theta_{m_k}(\omega) } 
    &\leq \sum_{k=1}^\infty \sum_{j=m_k}^{n_k - 1} \norm{ \Theta_{j+1}(\omega) - \Theta _j (\omega) }
    = \sum_{k=1}^\infty \sum_{j=m_k}^{n_k - 1} \rbr*{ \gamma_j \norm{ \fG^j ( \Theta_j (\omega) , \omega ) } }
    \\ 
    &\leq \fD  \br*{ \sum_{k=1}^\infty \sum_{j=m_k}^{n_k - 1} \gamma_j } < \infty.
    \end{split}
\end{equation}
Moreover, \nobs that the fact that $\sup\nolimits_{n \in \N_0} \norm{ \Theta_n ( \omega) } \leq \fC$ and \cref{lem:realization:lip} show that there exists $\scrL \in \R$ which satisfies for all
$ m, n \in \N_0$ that $\abs{ \cL  ( \Theta_m ( \omega) ) - \cL  ( \Theta_n ( \omega) ) } \leq \scrL \norm{ \Theta_m ( \omega) - \Theta_n ( \omega) }$.
This and \cref{theo:sgd:proof2} demonstrate that
\begin{equation}
    \limsup_{k \to \infty} \abs{ \cL  ( \Theta_{n_k} ( \omega) ) - \cL  ( \Theta_{m_k} ( \omega) ) } \leq \limsup_{k \to \infty} \rbr[\big]{ \scrL  \norm{ \Theta_{n_k} ( \omega ) - \Theta_{m_k} ( \omega ) } } = 0.
\end{equation}
Combining this and the fact that for all $k \in \N_0$ it holds that $\cL ( \Theta_{n_k} (\omega ) ) < \varepsilon < 2 \varepsilon < \cL ( \Theta_{m_k} ( \omega ) )$ ensures that
\begin{equation}
    0 < \varepsilon \leq \inf_{k \in \N} \abs{ \cL ( \Theta_{n_k} (\omega ) ) - \cL ( \Theta_{m_k} ( \omega ) ) } \leq \limsup_{k \to \infty} \abs{ \cL ( \Theta_{n_k} ( \omega ) ) - \cL ( \Theta_{m_k} ( \omega ) ) } = 0.
\end{equation}
This contradiction proves that $\limsup_{n \to \infty} \cL ( \Theta_n ( \omega)) = 0$. This and the fact that $\P ( A ) = 1 $ establish \cref{theo:sgd:item2}.
Next \nobs that \cref{theo:sgd:item1} and the fact that $\cL$ is continuous show that there exists $\scrC \in \R$ which satisfies that $\P \rbr*{  \sup_{n \in \N_0} \cL ( \Theta_n ) \leq \scrC  } = 1 $. This, \cref{theo:sgd:item2}, and the dominated convergence theorem establish \cref{theo:sgd:item3}.
\end{cproof}

\begin{cor} \label{cor:sgd:norm}
Assume \cref{setting:sgd}, let $\bfA \in \R$ satisfy $\bfA = \max \{ \bfa , \abs{\xi }, d \}$, assume $\sum_{n=0}^\infty \gamma_n = \infty$, and assume for all $n \in \N_0$ that $\P  \rbr*{ 18 \bfA ^5 \gamma_n  \leq ( \norm{\Theta_0} + 1 )^{-2} } = 1$. Then
\begin{enumerate} [label=(\roman*)]
\item \label{cor:sgd:norm:item1} there exists $\fC \in \R$ such that $\P  \rbr*{  \sup_{n \in \N_0}  \norm{ \Theta_n } \leq \fC  } = 1$,
    \item \label{cor:sgd:norm:item2} it holds that $\P  \rbr*{  \limsup_{n \to \infty} \cL  ( \Theta_n ) = 0  } = 1$, and
    \item \label{cor:sgd:norm:item3} it holds that $\limsup_{n \to \infty} \E [ \cL  ( \Theta_n ) ] = 0$.
\end{enumerate}
\end{cor}
\begin{cproof} {cor:sgd:norm}
Observe that \cref{prop:lyapunov:norm} proves that it holds $\P$-a.s.\ that
\begin{equation}
   \bfa^2 (d+1) V ( \Theta _ 0 ) + 1
    \leq 3 \bfa^2 (d+1) \norm{\Theta_0}^2 + 8 \xi  ^2 \bfa^2 (d+1) + 1  .
\end{equation}
The fact that $\bfA \geq d \geq 1$ hence shows that it holds $\P$-a.s.\ that
\begin{equation}
    \bfa^2 (d+1) V ( \Theta _ 0 )  + 1\leq  6 \bfA ^3 \norm{\Theta_0}^2 + 16 \bfA ^5 + 1  \leq 17 \bfA ^5 (\norm{\Theta_0}^2 + 1) \leq 17 \bfA ^5 (\norm{\Theta_0} + 1)^2. 
\end{equation}
This and the assumption that for all $n \in \N_0$ it holds that $\P  \rbr*{ 18 \bfA ^5 \gamma_n  \leq ( \norm{\Theta_0} + 1 )^{-2} } = 1$ ensure that for all $n \in \N_0$ it holds $\P$-a.s.\ that
\begin{equation}
    \gamma_n  \rbr{\bfa^2 (d+1) V ( \Theta _ 0 ) \allowbreak + 1} \leq 17 \bfA ^5 \gamma_n (\norm{\Theta_0} + 1)^2 \leq \tfrac{17}{18} < 1.
\end{equation}
\cref{theorem:sgd} hence establishes items \ref{cor:sgd:norm:item1}, \ref{cor:sgd:norm:item2}, and \ref{cor:sgd:norm:item3}.
\end{cproof}

\subsection{A {\sc Python} code for generalized gradients of the loss functions}
\label{subsection:python:code}

In this subsection we include a short illustrative example {\sc Python} code for the computation of appropriate generalized gradients of the risk function.
In the notation of \cref{setting:sgd} we assume in the {\sc Python} code in \cref{list:python} below that $d=1$, $\width = 3$, $\fd = 10$, $\phi = (-1, 1, 2, 2, -2, 0, 1, -1, 2, 3) \in \R^{10}$, $\xi = 3$, $\omega \in \Omega$, and $X^{1,1}(\omega) = 2$. \Nobs that in this situation it holds that $\w{\phi}_{1,1} X^{1,1} ( \omega ) + \b{\phi}_1 = \w{\phi}_{2,1} X^{1,1} ( \omega ) + \b{\phi}_2=0$. 
\cref{listing:output} presents the output of a call of the {\sc Python} code in \cref{list:python}.
\cref{listing:output} illustrates that the computed generalized partial derivatives of the loss with respect to $\w{\phi}_{1,1}$, $\w{\phi}_{2,1}$, $\b{\phi}_1$, $\b{\phi}_2$, $\v{\phi}_1$, and $\v{\phi}_2$ vanish. 
\Nobs that \cref{eq:def:sgd} and the fact that $\indicator{I_1^\phi} (X^{1,1}(\omega)) = \indicator{I_2^\phi} ( X^{1,1}(\omega) ) = 0$ prove that the generalized partial derivatives of the loss with respect to $\w{\phi}_{1,1}$, $\w{\phi}_{2,1}$, $\b{\phi}_1$, $\b{\phi}_2$, $\v{\phi}_1$, and $\v{\phi}_2$ do vanish.

\begin{lstlisting}[language=Python, caption=Generalized gradients of the loss functions using {\sc TensorFlow}, label=list:python] 
import tensorflow as tf
from tensorflow.python.framework.ops import disable_eager_execution

disable_eager_execution()

# batch size = 1
inputs = tf.compat.v1.placeholder(shape=(1, 1), dtype=tf.float64)
xi = 3

# first layer with constant initialization \R -> \R^3
w = tf.compat.v1.Variable(name='w', initial_value=[[-1., 1., 2.]], dtype=tf.float64, trainable=True)
b = tf.compat.v1.Variable(name='b', initial_value=[2., -2., 0.], dtype=tf.float64, trainable=True)

# second layer with constant initialization \R^3 -> \R
v = tf.compat.v1.Variable(name='v', initial_value=[[1.], [-1.], [2.]], dtype=tf.float64, trainable=True)
c = tf.compat.v1.Variable(name='c', initial_value=[3], dtype=tf.float64, trainable=True)

output = tf.matmul(tf.nn.relu(tf.matmul(inputs, w) + b), v) + c

loss = tf.reduce_mean((output - xi) ** 2)

gradw = tf.compat.v1.gradients(loss, w)
gradb = tf.compat.v1.gradients(loss, b)
gradv = tf.compat.v1.gradients(loss, v)
gradc = tf.compat.v1.gradients(loss, c)

with tf.compat.v1.Session() as sess:

    sess.run(tf.compat.v1.global_variables_initializer())

    gradw = sess.run(gradw, feed_dict={inputs: [[2.]]})
    print('gradient with respect to w: ', gradw)
    gradb = sess.run(gradb, feed_dict={inputs: [[2.]]})
    print('gradient with respect to b: ', gradb)
    gradv = sess.run(gradv, feed_dict={inputs: [[2.]]})
    print('gradient with respect to v: ', gradv)
    gradc = sess.run(gradc, feed_dict={inputs: [[2.]]})
    print('gradient with respect to c: ', gradc)
\end{lstlisting}
\begin{lstlisting} [caption = Result of the {\sc Python} code in \cref{list:python}, label = listing:output]
gradient with respect to w:  [array([[ 0.,  0., 64.]])]
gradient with respect to b:  [array([ 0.,  0., 32.])]
gradient with respect to v:  [array([[ 0.],
       [ 0.],
       [64.]])]
gradient with respect to c:  [array([16.])]
\end{lstlisting}

\subsection*{Acknowledgments}
The authors would like to thank Benno Kuckuck and Sebastian Becker for their helpful assistance and suggestions.
This work has been funded by the Deutsche Forschungsgemeinschaft (DFG, German Research Foundation) under Germany’s Excellence Strategy EXC 2044-390685587, Mathematics Münster: Dynamics-Geometry-Structure.


\end{document}